\documentclass{article}
\usepackage[utf8]{inputenc}

\usepackage{amsthm,amsmath,stmaryrd,bbm,amsfonts,amssymb,mathabx,graphicx,subcaption,hyperref}
\usepackage{color}
\newcommand{\po}{\left(}
\newcommand{\pf}{\right)}
\newcommand{\co}{\left[}
\newcommand{\cf}{\right]}
\newcommand{\cco}{\llbracket}
\newcommand{\ccf}{\rrbracket}
\newcommand{\R}{\mathbb R}

\newcommand{\N}{\mathbb N} 
\newcommand{\dd}{\text{d}}
\newcommand{\na}{\nabla}
\newcommand{\1}{\mathbbm{1}}
\newcommand{\Id}{\mathrm{Id}}

\newcommand{\bb}{\mathbf{b}}

\newcommand{\norm}[1]{\left\vvvert #1\right\vvvert}

\newtheorem{thm}{Theorem}
\newtheorem{prop}{Proposition}
\newtheorem{lem}{Lemma}
\newtheorem{assu}{Assumption}
\newtheorem{exemple}{Example}
\newtheorem{rem}{Remark}
\title{Asymptotic expansion of the invariant measure for Markov-modulated ODEs at high  frequency}
\author{Pierre Monmarché\footnote{Sorbonne Université,
4 place Jussieu 75005 Paris (France)}, Édouard Strickler\footnote{Université de Lorraine, CNRS, Inria, IECL, F-54000 Nancy (France)}}

\begin{document}

\maketitle

\begin{abstract}
 We consider time-inhomogeneous ODEs whose parameters are governed by an underlying ergodic Markov process. When this underlying process is accelerated by a factor $\varepsilon^{-1}$, an averaging phenomenon occurs and the solution of the ODE converges to a deterministic ODE as $\varepsilon$ vanishes. We are interested in cases where this averaged flow is globally attracted to a point. In that case, the equilibrium distribution of the solution of the ODE converges to a Dirac mass at this point. We prove an asymptotic expansion in terms of $\varepsilon$ for this convergence, with a somewhat explicit formula for the first order term. The results are applied in three contexts:   linear Markov-modulated ODEs, randomized splitting schemes, and   Lotka-Volterra models in random environment. In particular, as a corollary, we prove the existence of two matrices whose convex combinations are all stable but such that, for a suitable jump rate, the top Lyapunov exponent of a Markov-modulated linear ODE switching between these two matrices is positive.
\end{abstract}

\section{Introduction}

Let $\mathcal M$ be a $d$-dimensional compact $\mathcal C^\infty$ manifold. Markov-modulated ODEs on $\mathcal M$ are dynamical systems $(x(t))_{t\geqslant 0}$ solution to
\begin{equation}
    \label{eq:EDOswitch}
    x'(t) = F_{\sigma(t)} \po  x(t) \pf 
\end{equation}
where $(\sigma(t))_{t\geqslant 0}$ is a Markov process on some space $\mathcal S$ and, for all $s\in\mathcal S$, $F_s$ is a vector field on $\mathcal M$. 

We are interested in the high frequency regime, namely when $\sigma(t)$ is replaced in \eqref{eq:EDOswitch} by $\sigma(t/\varepsilon)$ for some small $\varepsilon$. Provided $\sigma$ is ergodic with respect to some probability measure $\pi$, the modulated process $x(t)$ is known to converge, as $\varepsilon$ vanishes, to the solution of the averaged ODE
\begin{equation}
    \label{eq:EDOaveraged}
x'(t) = \bar F(x(t)) \,,\quad\text{where}\quad \bar F(x) = \int_{\mathcal S} F_s(x) \pi(\dd s)\,.
\end{equation}
In the case where this averaged ODE has a unique global attractor $\bar x$, we expect the solution of \eqref{eq:EDOswitch} to be close to $\bar x$ for $t$ large and $\varepsilon$ small. In other words,   the invariant measure of the Markov process $(x(t),\sigma(t/\varepsilon))_{t\geqslant 0}$  converges, as $\varepsilon$ vanishes, to $\delta_{\bar x} \otimes \pi$. The goal of the present work is to provide an infinitesimal expansion in terms of $\varepsilon$ for this invariant measure.  This is obtained by combining a similar expansion for the law of the process for a fixed $t>0$ (following \cite{PTW2012}) with a long-time convergence result for the limit process.  The structure of the proof follows the work of Talay and Tubaro \cite{TalayTubaro} who established a similar expansion for the invariant measure of Euler-Maruyam schemes for  diffusion processes (the averaging estimates for a fixed $t>0$ being in that case replaced by  finite-time discretization error expansions).

One of our main motivation is to get the first-order term in the expansion of the top Lyapunov exponent of systems of switched linear ODE \cite{Benaim2014Stability,gurvits07,chitour2021}, as detailed in Section~\ref{sec:appli_Lyapunov}. Our result also has applications for instance in some population models \cite{MZ17,BL16,B18} or 
 random splitting numerical schemes \cite{RandomSplitting}, cf. Section~\ref{sec:applications}. More generally, Markov-modulated ODEs form a flexible class of Markov processes which appear in a variety of models, often to describe systems evolving in a randomly fluctuating environment, as in finance \cite{finance}, biology \cite{hatzikirou2021novel} or fiability \cite{fiabilite}. It is related to random ODEs, see e.g. \cite{arnold2006lyapunov,chueshov} and references within.
 
 In the specific context of Piecewise-Deterministic Markov processes (PDMP; namely, when $\sigma$ is a Markov chain on a finite set), the question of fast averaging has been addressed in~\cite{FGR09} where a large deviation principle is proven, in~\cite{PTW2012} where the expansion of the law of the process at fixed time $t$ is given, or in~\cite{benaim2019} where the convergence of the invariant measure of the Markov process $(x(t), \sigma(t/\varepsilon))_{t \geq 0}$ is proven. In the recent paper~\cite{goddard2023study}, the authors deal with the case where $\sigma$ is not ergodic.  
 
 The paper is organized as follows. In the remaining of this introduction, we introduce our general notations and assumptions. Our main result, Theorem~\ref{thm:expansion_mu}, is stated and proven in Section~\ref{sec:proof_Markov}. Section~\ref{sec:applications} is devoted to exemples of applications. Finally, an appendix gathers the proofs of some intermediary results.
 
 \bigskip

To conclude this section, let us clarify our settings.  Denote by $(\mathcal Q_t)_{t\geqslant 0}$ the semi-group associated to $(\sigma_t)_{t\geqslant 0}$, namely
 \[\mathcal Q_t f(\sigma) = \mathbb E_{\sigma}\po f(\sigma_t)\pf\,,\]
 and by $Q$ its infinitesimal generator.

Denote by $\mathcal A$ the set of bounded measurable functions from $\mathcal M\times\mathcal S$ to $\R$ which are $\mathcal C^\infty$ in their first variable $x$ and such that all their derivatives in $x$ are bounded over $\mathcal M\times\mathcal S$, and which moreover are such that for  all $x\in \mathcal M$ and all multi-index $\alpha$, $s\mapsto \partial_x^\alpha f(x,s)$ is in the domain of $Q$.
 We consider on $\mathcal A$ the norms 
\[\norm{f}_{j }  =  \sum_{|\alpha|\leqslant j }    \| \partial^\alpha_x   f \|_{\infty}\,.\]

\begin{assu}\label{assu:settings}
We assume that:
\begin{enumerate}
    \item Almost surely, the ODE \eqref{eq:EDOswitch} is well-defined for all times, with values in $ \mathcal M$. Moreover, the coordinates of $(x,s) \mapsto F_s(x) $ are in $\mathcal A$.
    \item The semigroup $(\mathcal Q_t)_{t\geqslant 0}$ admits a unique invariant probability measure $\pi$. Moreover,
     there exist $C ,\gamma  >0$ such that 
 \begin{equation}\label{eq:ergodiciteQ_t}
\|\mathcal Q_t f- \pi f \|_{\infty} \leqslant C  e^{-\gamma  t} \|f\|_{\infty }
\end{equation}
for all bounded measurable $f$ on $\mathcal S$ and all $t\geqslant  0$. 
\end{enumerate}
\end{assu}

The uniform exponential ergodicity \eqref{eq:ergodiciteQ_t} implies that the operator $Q^{-1}$ given by
\begin{equation}\label{eq:Q-1}
    Q^{-1} f := \int_0^\infty \mathcal Q_t\po \pi f -  f\pf \dd t
\end{equation}
is well-defined for all $f\in\mathcal A$, with $Q^{-1} f\in \mathcal A$ and $\norm{Q^{-1} f}_j  \leqslant C /\gamma \norm{f}_{j }$ for all $j\in\N$. Moreover, using that $\partial_t \mathcal Q_t f = Q \mathcal Q_t f = \mathcal Q_t Q f $ for all $t\geqslant 0$, we get that $QQ^{-1} f =   Q^{-1}Q f = f - \pi f$ for  $f\in\mathcal A$, namely $Q^{-1}$ is the pseudo-inverse of $Q$. 

\begin{exemple}\label{exemple:Q=pi-I}
A case of interest is given by $Qf = \pi f - f$, which corresponds to $\sigma(t) = Y_{N_t}$ where $(N_t)_{t\geqslant 0}$ is a standard Poisson process with intensity $1$, $Y_0=\sigma(0)$ and $(Y_k)_{k\geqslant 1}$ is an i.i.d. sequence of random variables distributed according to $\pi$. It is readily checked that, in that case, $\mathcal Q_t f = e^{-t} f + (1-e^{-t}) \pi f$ and then $Q^{-1} = Q$.
\end{exemple}

\begin{exemple}
\label{exemple:Q=matrix}
We will be particularly concerned with the case when $\mathcal{S}$ is a finite state space and $(\sigma_t)_{t \geq 0}$ is an irreducible continuous-time Markov chain on $\mathcal{S}$, with transition rate matrix $Q$.
In that case, $Q^{-1}$ is the \emph{group inverse} of $Q$, defined as the unique matrix $X$ solution\footnote{such a matrix exists and is unique since $Q$ has index $1$, meaning that $ Q^2$ and $Q$ have the same rank, because $0$ is a simple eigenvalue of $Q$. See Chapter 4 in \cite{BIG} for more details.} to
\begin{equation}
\label{eq:groupinverse}
QXQ= Q, \quad XQX=X, \quad XQ = QX.    
\end{equation} In the particular case when the cardinal of  $\mathcal{S}$ is 2, then $Q$ can be written as
\[
Q = \begin{pmatrix}
-p & p\\
q & -q
\end{pmatrix},
\]
for some $p, q > 0$. Moreover, one can easily check, using $Q^2 = - (p+q) Q$ that $X =  \frac{1}{(p+q)^2}Q$ satisfies \eqref{eq:groupinverse}, and therefore $Q^{-1} =  \frac{1}{(p+q)^2}Q$.
\end{exemple}

\section{High-frequency  expansion}\label{sec:proof_Markov}

The process $(X(t),\sigma(t/\varepsilon))_{t\geqslant 0}$ is a  Markov process on $\mathcal M\times\mathcal S$ with generator $L = L_c + \frac1\varepsilon Q$, where $Q$ is  seen as a an operator on functions on $\mathcal M\times \mathcal S$ acting only on the second variable and
\[
L_c f(x,s) = F_s(x) \cdot \na_x f(x,s)\,,
\]
for all $f\in\mathcal A$. For $t\geqslant 0$, denote by $P_t^\varepsilon$ the associated Markov semigroup, namely
\[P_t^\varepsilon f(x,s) = \mathbb E_{x,s} \po f(X(t),\sigma(t/\varepsilon))\pf\,.\]

We start by stating an expansion in $\varepsilon$ of $P_t^\varepsilon$ for a fixed $t$ (the proof is postponed to the Appendix). This is essentially the result (and proof)  of \cite{PTW2012} but written in a dual form and with  more explicit functional settings. Moreover,  \cite{PTW2012} only considers the case where $\mathcal S$ is a finite set, but it allows the jump rates of $\sigma$ to depend on $x$.

%\[
%\| \nabla^j f \| = \sum_{ \alpha \in \cco 1, d \ccf^j} \| \partial_{\alpha} f \|_{\infty}
%\]

%\[
%\| \partial_{\alpha} f \|_{\infty} = \max_{(x,i)} | \partial_{\alpha} f_i(x) |
%\]
\begin{prop}
\label{prop:dev-semigroup}
There exist two families of operators $P_t^{(k)},S_t^{(k)}$ indexed by $t\geq 0,k\in\N$, acting on $\mathcal A$, with the following properties.
\begin{itemize}
\item    For all $k,j  \in\N$ and $T>0$, there exists $C>0$ such that for all $f\in\mathcal A$,
\begin{equation}
\label{eq:boundsP}
\sup_{t\in[0,T]}\norm{ P_{t}^{(k)} f}_{j } \leq C \norm{f}_{j+2k }\,.
\end{equation}
    \item For all $k,j  \in\N$, there exist $C,\gamma>0$ such that for all $t\geq 0$ and all $f\in\mathcal A$,
    \begin{equation}
    \label{eq:boundsS}
\norm{ S_{t}^{(k)} f}_{j }  \leq C e^{- \gamma t}\norm{f}_{j+2k  }\,.
\end{equation}
    \item For all $n\in\N$, $T\geq 0$, $j\in\N$, there exists $C>0$ such that for all $\varepsilon >0$ and $f\in\mathcal A$,  the remainder $R_{n,t}^\varepsilon f$ defined by
    \begin{equation}
\label{eq:dev-pt}
R_{n,t}^\varepsilon f = P_t^{\varepsilon}f -\sum_{k=0}^n \varepsilon^k P_t^{(k)}f -  \sum_{k=0}^n \varepsilon^k S_{\frac{t}{\varepsilon}}^{(k)} f,
\end{equation}
satisfies
\begin{equation}
\label{eq:boundsR}
\norm{ R_{n,t}^{\varepsilon}f }_j \leq C \varepsilon^{n+1}\norm{f}_{j+3+2n}\,.
\end{equation}
\item The first operators are given by
\begin{equation}\label{eq:Pt0St0}
P_t^{(0)} f(x,s) = \pi f(\bar{\varphi}_t(x)), \quad S_{t}^{(0)}f(x,s) =   \mathcal Q_t f(x,s) - \pi f(x) 
\end{equation}
and
\begin{equation}
P_t^{(1)}f(x,s) = \bb_1(t,x,s) + \int_0^{+ \infty} \pi (L_c S_r^{(0)} f)(x) dr + \int_0^t \pi( L_c \bb_1)(r,  \bar{\varphi}_{t-r}(x)) dr,
    \label{eq:Pt1}
\end{equation}
where
\begin{equation}\label{eq:bb1}
   \bb_1 (t, x,s) = Q^{-1} ( \partial_t P_t^{(0)}f - L_c P_t^{(0)}f) (x,s). 
\end{equation}

\end{itemize}
\end{prop}

Next, to address the question of the long-time behaviour of the process, we work under the following condition: 

\begin{assu}\label{assu:fast}
There is a globally attractive $\bar x\in\mathcal M$ for  the averaged flow $\bar F$ given in \eqref{eq:EDOaveraged}, and for all $j\in\N$, there exist $C,a>0$ such that for all $t\geqslant 0$ and $f\in\mathcal C^\infty(\mathcal M)$,
\[\norm{ f \circ \varphi_t - f \po \bar x\pf}_j  \leqslant C e^{-at}\norm{f}_{j+1}\,,\]
where $\varphi$ is the flow associated to $\bar F$.
\end{assu}

Under Assumption~\ref{assu:fast}, $\mu_0=\delta_{\bar x}\otimes\pi$ is the unique invariant measure of $P_t^{(0)}$ given in \eqref{eq:Pt0St0}. Our main result is the following expansion in terms of $\varepsilon$ of any invariant measure of the process $(P_t^\varepsilon)_{t\geqslant 0}$, in the spirit of Talay-Tubaro expansions in terms of the step size for discretization schemes \cite{TalayTubaro}.

\begin{thm}
\label{thm:expansion_mu}
Under Assumptions~\ref{assu:settings} and \ref{assu:fast}, for all $f\in\mathcal A$ there exists real sequences $(c_k)_{k\geq 1} $ and $(M_k)_{k\geq 1} $ such that for all $n\in\N$ and any invariant measure $\mu_\varepsilon$ of  $(P_t^\varepsilon)_{t\geqslant 0}$,
\begin{equation}\label{eq:expansion}
   \left| \mu_\varepsilon f - \mu_0 f - \sum_{k=1}^n c_k \varepsilon^k \right| \leqslant M_n \varepsilon^{n+1}. 
\end{equation}
Moreover, 
\begin{equation}\label{eq:c1}
  c_1 = \pi L_c Q^{-1} \po       L_c   h - f  \pf (\bar x)   
\end{equation}
where, for $(x,s)\in\mathcal M\times \mathcal S$, 
\begin{equation}\label{eq:def_h}
h(x,s) = \int_0^\infty \pi \po f(\bar x) - f\po \varphi_r(x)\pf\pf dr \,.    
\end{equation}
\end{thm}
\begin{rem}
\label{rem:dim1}
In dimension 1, it is possible to get an alternative expression for the function $h$ appearing in formula~\eqref{eq:c1}. Indeed, for fixed $x \in \mathbb{R}$, with the change of variables $y = \varphi_r(x)$; we obtain, due to $\frac{d \varphi_r(x)}{dr} = \bar F ( \varphi_r(x))$,
\[
h(x,s) = \int_0^{+ \infty} \po \pi f(\bar x) - \pi f( \varphi_r(x)) \pf d r = \int_{\bar x}^x \frac{\pi f(y) - \pi f(\bar x)}{\bar F(y)}dy.
\]
We will use this expression in Section~\ref{subsec:LV}
\end{rem}
\begin{proof}[Proof of Theorem \ref{thm:expansion_mu}] The proof is divided in two parts. In the first one, we prove \eqref{eq:expansion}, and we give an expression of $c_1$ which depends on an arbitrary time $t>0$. In the second part, we obtain the announced expression of $c_1$ by letting $t$ vanish in this first expression.

\textbf{Step 1.}  Fix $t>0$, $f \in\mathcal A$ and let $\mu_\varepsilon$ be an invariant measure of $(P_s^\varepsilon)_{s\geqslant 0}$. By definition of the flow $\varphi_t$,
\[(P_t^{(0)})^n f(x,i) = \pi f\po \varphi_t^{\circ n}(x) \pf = \pi f\po \varphi_{tn}(x) \pf.\]
Hence,
\[
\norm{(P_t^{(0)})^n ( \mu_0 f - f)}_j = \norm{\pi f(\bar x) - \pi( f \circ  \varphi_{tn} )}_j \leq C e^{-at n} \norm{f}_{j+1}.
\]
As a consequence, the function $\Psi_t$ given by 
\begin{equation}
\label{eq:Psit}
\Psi_t = \sum_{n=0}^{\infty} (P_t^{(0)})^n (\mu_0 f -  f)   
\end{equation}
is well-defined, in $\mathcal A$ and such that, for all $j\geq 0$, $\norm{\Psi_t}_j \leq C_j\norm{f}_{j+1}$ for some $C_j>0$ independent from $f$. Moreover, 
\[
(P_t^{(0)} - \Id) \Psi_t  = \sum_{n=1}^{\infty} (P_t^{(0)})^n (\mu_0 f -  f) - \sum_{n=0}^{\infty} (P_t^{(0)})^n (\mu_0 f -  f) = f - \mu_0 f.
\]
Then, using this Poisson equation and that $\mu_\varepsilon$ is invariant for $P_t^\varepsilon$,
\[
\mu_\varepsilon f - \mu_0 f = \mu_\varepsilon (P_t^{(0)} - \Id) \Psi_t  = \mu_\varepsilon (P_t^{(0)} - P_t^\varepsilon) \Psi_t.
\]
To get the convergence of $\mu_\varepsilon f$ to $\mu_0 f$ at zeroth order, we simply bound, thanks to Proposition \ref{prop:dev-semigroup},
\[|\mu_\varepsilon f - \mu_0 f| \leqslant \|(P_t^{(0)} - P_t^\varepsilon) \Psi_t\|_\infty \ \leqslant C \varepsilon    \norm{\Psi_t}_3  \,.%  \underset{\varepsilon\rightarrow 0}\longrightarrow 0.
\]
To get the higher order expansion, write
\begin{align}
\mu_\varepsilon f - \mu_0 f &= \mu_\varepsilon (P_t^{(0)}  - P_t^{\varepsilon}) \Psi_t \nonumber\\
 &= -\mu_\varepsilon \po \sum_{k=1}^n \varepsilon^k P_t^{(k)} +  \sum_{k=0}^n \varepsilon^k S_{\frac{t}{\varepsilon}}^{(k)} + R_{n,t}^{\varepsilon}\pf \Psi_t.\label{eq:mueps-mu0}
\end{align}
From Proposition \ref{prop:dev-semigroup} and the bounds on $\norm{\Psi_t}_j$ for all $j\in\N$,
\[\left|\mu_\varepsilon \po \sum_{k=0}^n \varepsilon^k S_{\frac{t}{\varepsilon}}^{(k)} + R_{n,t}^{\varepsilon}\pf \Psi_t \right| \leqslant C_n \varepsilon^{n+1}\]%=\underset{\varepsilon\rightarrow0}{o}(\varepsilon^n).\]
for some $C_n$. Using that, from the convergence at order $0$,
\begin{equation*}
%\label{eq:muepsPtkPsi}
\left| \mu_\varepsilon P_t^{(k)} \Psi_t -  \mu_0 P_t^{(k)} \Psi_t\right| \leqslant C_{0,k} \varepsilon 
\end{equation*}
for all $k\in\N$ for some $C_{0,k}>0$,  we get first with \eqref{eq:mueps-mu0} at order $n=1$ that 
\begin{equation}\label{eq:c1Psit}
    \left| \mu_\varepsilon f - \mu_0 f + \varepsilon \mu_0 P_t^{(1)} \Psi_t \right|\leqslant C_{1,0} \varepsilon^2
\end{equation}
% \underset{\varepsilon\rightarrow0}{o}(\varepsilon).\]
for some $C_{1,0}>0$. We can thus apply this first order expansion with $f$ replaced by $P_t^{(k)} \Psi_t$ to get
\[\left |\mu_\varepsilon P_t^{(k)} \Psi_t - c_{0,k}-c_{1,k}\varepsilon\right| \leqslant C_{1,k} \varepsilon^2\]% + \underset{\varepsilon\rightarrow0}{o}(\varepsilon)\]
for some  $c_{0,k},c_{1,k},C_{1,k}$. Hence, considering the expansion \eqref{eq:mueps-mu0} at order $n=2$, we get
\[\left| \mu_\varepsilon f - \mu_0 f - \varepsilon c_{0,1} - \varepsilon^2 c_{1,1} - \varepsilon^2 c_{0,2} \right| \leqslant C_{2,0} \varepsilon^{3}\,,\]% \underset{\varepsilon\rightarrow0}{o}(\varepsilon^2).\]
for some $C_{2,0}>0$. Again, this can be used with $f$ replaced by $P_t^{(k)}\Psi_t$, and a straightforward induction on $n$ concludes the proof of \eqref{eq:expansion}. In particular, from \eqref{eq:c1Psit}, we see that   for all $t>0$, $c_1$ can be written as
\begin{equation*}%\label{eq:c1_v1}
c_1 = -\mu_0 P_t^{(1)} \Psi_t \,.
\end{equation*}

\textbf{Step 2.} The goal is now to let $t$ vanish in this last expression. For a fixed $t>0$,  denote by $\bb_1^\Psi$  the function given by \eqref{eq:bb1} but with $f$ replaced by $\Psi_t$. Then
\[P_t^{(1)}\Psi_t(x,s) = \bb_1^\Psi(t,x,s) + \alpha(x) + \beta(t,x)\]
with 
\[\alpha(x) =  \int_0^{+ \infty}    \pi (L_c S_r^{(0)} \Psi_t) (x) dr \,,\qquad \beta(t,x)= \int_0^t      \pi( L_c \bb_1^\Psi)(r, \bar\varphi_{t-r}(x)) dr\,. \]
First, by definition of $Q^{-1}$, the average of $\bb_1^\Psi$ with respect to $\pi$, hence with respect to $\mu_0$, is zero. At this stage, we have obtained that $c_1=-\mu_0\alpha-\mu_0 \beta$. 
%Hence,
%\[c_1 = \mu_0 P_t^{(1)}\Psi_t =  \theta_1^\Psi(t,\bar x) =  \theta_1^\Psi(0,\bar x)  + \int_0^t \pi L_c \bb_1^\Psi(s,\bar x)ds. \]
Recalling that  $S_r^{(0)} f = \mathcal Q_r(f-\pi f)$, from
\[  \Psi_t(x,s) = \pi f(\bar x)  - f(x,s) + \sum_{n=1}^\infty \pi \po f(\bar x) - f(\varphi_{tn}(x))\pf, \]
integrating with respect to $\pi$, we get
\[\na_x S_r^{(0)} \Psi_t(\bar x,s) = \na_x \mathcal Q_r\po  \Psi_t - \pi \Psi_t \pf (\bar x,s) = \mathcal \na_x \mathcal Q_r \po   \pi f -    f \pf (\bar x,s) . \]
As a consequence,
\begin{align*}
\mu_0(\alpha)  = \alpha(\bar x) &= \int_0^\infty \pi L_c S_r^{(0)} \Psi_t(\bar x) dr\\
%&= \int_0^\infty \pi \po F(\bar x) \cdot \mathcal Q_r \po \na_x \pi f(\bar x) - \na_x  f(\bar x,s)\pf \pf  d r\\
&=  \int_{\mathcal{S}}  F_s(\bar x) \cdot \co \int_0^\infty \na_x Q_r \po  \pi f  -  f \pf dr \cf(\bar x, s) \pi(ds)\\
&=  \int_{\mathcal S} F_s(\bar x) \cdot  \na_x  Q^{-1}   f(\bar x,s)       \pi(ds)\\
%& = \pi \po F\cdot Q^g\po \na f- \na\pi f \1 \pf \pf (\bar x)\\
& = \pi L_c Q^{-1} f(\bar x).
\end{align*}
 Besides, 
\begin{align*}
    \mu_0 (\beta) &=   \int_0^t \pi L_c \bb_1^\Psi(r,\bar x)dr \\
    &= \pi \co L_c \int_0^t \bb_1^\Psi(r,\cdot )dr\cf (\bar x)\\
     &= \pi \co L_c  \int_0^t Q^{-1} \po \partial_r P_r^{(0)} \Psi_t  - L_c P_r^{(0)} \Psi_t \pf  dr\cf (\bar x)\\
     &= -\pi \co L_c Q^{-1} \int_0^t    L_c P_r^{(0)} \Psi_t  dr\cf (\bar x)\,,
\end{align*}
where we used that $\partial_r P_r^{(0)} \Psi_t$ does not depend on  $s$ and is thus in the kernel of $Q^{-1}$. Hence,
\begin{align*}
    \int_0^t \pi L_c \bb_1^\Psi(s,\bar x)ds 
     &= -\pi \co L_c Q^{-1} \frac1t \int_0^t   L_c   P_r^{(0)} (t\Psi_t)    dr\cf (\bar x)\\
     & \underset{t\rightarrow 0}\longrightarrow -\pi \co L_c Q^{-1}     L_c (\pi h) \cf (\bar x),
\end{align*}
where we used that $P_0^{(0)} f= \pi f $ and that $t\Psi_t \rightarrow h$. Indeed, thanks to Assumption~\ref{assu:fast},
\begin{eqnarray*}
 t \Psi_t(x) &= & t\pi f(\bar x)- tf(x,s) +t \sum_{n=1}^\infty \pi \po f(\bar x) - f(\varphi_{tn}(x))\pf \\
  & \underset{t\rightarrow0}\longrightarrow &  \int_0^\infty \pi \po f(\bar x) - f\po \varphi_r(x)\pf\pf dr  \ = \ h(x)\,,
\end{eqnarray*}
and the convergence $t\Psi \rightarrow h$ holds in all norms $\norm{\cdot}_j$. The proof of \eqref{eq:c1} is concluded by letting $t$ vanish in the equality $c_1 = -\mu_0 \alpha - \mu_0 \beta$ (since $c_1$ is independent from $t$).

\end{proof}

\section{Applications}\label{sec:applications}

\subsection{Top Lyapunov exponent for cooperative linear Markov-modulated ODEs}\label{sec:appli_Lyapunov}

In this section, we consider linear Markov-modulated ODEs on $\R^d$ of the form
\begin{equation}\label{eq:EDOrandom}
 z'(t) = A\po \sigma(t/\varepsilon )\pf z(t)    
\end{equation}
where, for all $\sigma \in\mathcal S$, $A\po \sigma\pf $ is a $d\times d$ matrix. We work under Assumption~\ref{assu:settings}  and write $\bar A = \int_{\mathcal S} A(s)\pi(ds)$. Moreover, we focus  on the settings of \cite{NoteTopLyapunov}, given as follows:
\begin{assu}\label{assu:Lyapunov}
\begin{enumerate}
    \item The Markov process $(\sigma(t))_{t\geqslant 0}$ is Feller.
    \item For all $s\in\mathcal S$, $A(s)$ is a cooperative matrix, in the sense that its off-diagonal coefficients   are non-negative.
    \item The averaged matrix $\bar A$ is irreducible in the sense that for all $i,j\in\cco 1,d\ccf$ there exists a path $i=i_0,\dots,i_q=j$ with $\bar A_{i_{k-1},i_{k}} >0$ for all $k \in \cco 1,q \ccf$.
\end{enumerate}
\end{assu}
The fact that the matrices are cooperative implies that $\R_+^d$ is fixed by \eqref{eq:EDOrandom} for $t\geqslant 0$ and that Assumption~\ref{assu:fast} holds. We decompose solutions of \eqref{eq:EDOrandom} on $\R_+^d$ as $z(t) = \rho(t)\theta(t)$ where $\rho(t) = \1\cdot z(t)>0$ with  $\1=(1,\dots,1) \in \R^d$ and $\theta(t) = z(t)/\rho(t) \in \Delta = \{x\in\R_+^d, x_1+\dots+x_d=1\}$. The ODE \eqref{eq:EDOrandom} is then equivalent to
\[\rho'(t) = (\1\cdot A\po \sigma(t/\varepsilon)\pf  \theta(t)) \rho(t)\,,\qquad \theta'(t) = F_{\sigma(t/\varepsilon)}(\theta(t))\]
with
\[F_s(\theta) = A(s)\theta - (\1\cdot A(s)\theta )\theta\,.\]
It is proven in \cite{NoteTopLyapunov} that, under Assumption~\ref{assu:Lyapunov},  the Markov process $(\theta(t),\sigma(t/\varepsilon))_{t\geqslant 0}$ admits a unique invariant measure $\mu_\varepsilon$ on $ \Delta \times \mathcal S$ and that for all initial conditions $z\in \R_+^d\setminus\{0\}$, almost surely, 
\begin{equation}
    \lim_{t \to \infty} \frac{\rho(t)}{t} = \Lambda_\varepsilon,
\end{equation}
where
\begin{equation}
\label{eq:Lyapepsilon}
    \Lambda_\varepsilon := \int_{ \Delta \times \mathcal S} \1 \cdot A(s) \theta\, \mu_\varepsilon( \dd \theta, \dd s)
\end{equation}
is called the top Lyapunov exponent of the process. Denoting by $\lambda_{\max}(\bar A)$ the principal eigenvalue of $\bar A$ (i.e., with maximal real part), Proposition 4  in \cite{NoteTopLyapunov} entails that 
\[
\lim_{ \varepsilon \to 0} \Lambda_ \varepsilon = \lambda_{\max}( \bar A)\,.
\]
From Theorem~\ref{thm:expansion_mu} applied to the function $f :(\theta,s) \mapsto \1 \cdot A(s) \theta $, we get an expansion of $\Lambda_\varepsilon$ for small $\varepsilon$.

\begin{prop}\label{prop:lyapunov}
Under Assumptions~\ref{assu:settings} and \ref{assu:Lyapunov}, there exists a sequence $(c_k)_{k\geqslant 1}$ of real numbers such that, for all $n\geqslant 1$,
\[
\Lambda_\varepsilon = \lambda_{\max}( \bar A) + \sum_{k=1}^n c_k \varepsilon^k + \underset{\varepsilon\rightarrow 0}o(\varepsilon^n)\,.
\]
Moreover, denoting by $\bar x$ and $\bar y$, respectively, the right and left eigenvectors of $\bar A$ associated with $\lambda_{\max}( \bar A)$ and such that $\bar x\in\Delta$ and $\bar x \cdot \bar y = 1$, it holds
\[
c_1 =  \int_{\mathcal{S}} \bar y^{\intercal} Q^{-1}(A)(s) \po  \bar x \bar y^{\intercal} - I \pf A(s) \bar x \pi( \dd s)\,.
\]

\end{prop}

\begin{proof}
 First, considering $h$ given by \eqref{eq:def_h}, we claim that for all $(x, s) \in \Delta \times \mathcal{S},$
\begin{equation}
    h(x,s) = - \ln( \bar y \cdot x).
    \label{eq:hLyap}
\end{equation}
Indeed, note that, for all $x \in \mathcal{M}$,
\begin{align*}
  \pi f\po \varphi_r(x)\pf & = \frac{\1 \cdot \bar A e^{r \bar A}x}{\1_d \cdot e^{r \bar A}x}  \\
  & = \frac{d}{dr} \ln( \1  \cdot e^{r \bar A}x)\,.
\end{align*}
Theorefore, using that $\bar x = \varphi_r(\bar x)$, one has
\begin{align*}
    \int_0^\infty \po \pi f\po \bar x \pf - \pi f \po \varphi_r(x) \pf \pf \dd r & =   \int_0^\infty \frac{d}{dr} \po \ln \po \frac{\1 \cdot e^{r \bar A}\bar x}{\1 \cdot e^{r \bar A}x} \pf \pf \dd r\\
    & = \lim_{r \to \infty} \ln \po \frac{\1 \cdot e^{r \bar A} \bar x}{\1 \cdot e^{r \bar A}x} \pf\\
    & =-  \ln ( \bar y \cdot x ),
\end{align*}
where we have used  Perron-Frobenius Theorem. 
From~\eqref{eq:hLyap}, we deduce that $L_c h(x,s) = - F_s(x) \cdot \frac{\bar y}{\bar y \cdot x}$.  Now, using that  $F_s(x) = A(s) x - f_s(x) x$, we get that 
\[
L_c h(x,s) - f(x,s)   = - \frac{A(s) x \cdot \bar y}{\bar y \cdot x},
\]
and then
\[
Q^{-1}(L_c h-f)(x,s) =- \frac{Q^{-1}(A)(s) x \cdot \bar y}{\bar y \cdot x}\,,
\]
from which we deduce that
\begin{align*}
\na_x Q^{-1}(L_c h-f)(\bar x,s) &=   \frac{Q^{-1}(A)(s) \bar x \cdot \bar y}{(\bar y \cdot \bar x)^2 }\bar y - \frac{[Q^{-1}(A)(s)]^{\intercal}\bar y}{\bar y \cdot \bar  x}  \\
&=     (Q^{-1}(A)(s) \bar x \cdot \bar y )\bar y - [Q^{-1}(A)(s)]^{\intercal}\bar y   \,,
\end{align*}
where we have used that $\bar x \cdot \bar y = 1$ As a consequence,
\begin{eqnarray*}
\lefteqn{F_s(\bar x) \cdot \na_x Q^{-1}(L_c h-f)(\bar x,s)}\\
&=& \po A(s) \bar x - (\1\cdot A(s) \bar x) \bar x \pf\cdot  \po    (Q^{-1}(A)(s) \bar x \cdot \bar y )\bar y - [Q^{-1}(A)(s)]^{\intercal}\bar y \pf \\
    &=&  (Q^{-1}(A)(s) \bar x \cdot \bar y ) A(s) \bar x \cdot  \bar y  - A(s) \bar x \cdot    [Q^{-1}(A)(s)]^{\intercal}\bar y \\
    & & -  (\1\cdot A(s) \bar x) (Q^{-1}(A)(s) \bar x \cdot \bar y ) \bar x \cdot \bar y  + (\1\cdot A(s) \bar x) \bar x \cdot   [Q^{-1}(A)(s)]^{\intercal}\bar y \,.
\end{eqnarray*}
Using that $\bar x \cdot \bar y =1$, we see that the last line is equal to 
\[-  (\1\cdot A(s) \bar x) Q^{-1}(A)(s) \bar x \cdot   \bar y  +  (\1\cdot A(s) \bar x) Q^{-1}(A)(s) \bar x \cdot \bar y    = 0 \,.\]
We end up with 
\begin{eqnarray*}
\lefteqn{F_s(\bar x) \cdot \na_x Q^{-1}(L_c h-f)(\bar x,s)}\\
&=&    (Q^{-1}(A)(s) \bar x \cdot \bar y ) A(s) \bar x \cdot  \bar y - Q^{-1}(A)(s)  A(s) \bar x \cdot    \bar y  \\
&=& \bar y^{\intercal}  Q^{-1}(A)(s)  \co      \bar x  \bar y^{\intercal} - I\cf  A(s)  \bar x  \,.
\end{eqnarray*}
Integrating with respect to $\pi$ leads to
\[
c_1 =  \int_{\mathcal{S}} \bar y^{\intercal} Q^{-1}(A)(s) \po \bar x \bar y^{\intercal} - I \pf A(s) \bar x \pi( \dd s)\,.
\]
\end{proof}

\begin{exemple}
If $\sigma$ is a Markov chain on $\mathcal{S}= \{0, 1\}$,  with matrices rates given by 
\begin{equation}
\label{eq:ratematrix2x2}
    Q =  \begin{pmatrix}
- p & p \\
1 -p & -(1 - p)
\end{pmatrix}\,,
\end{equation}
we get the following simple expression for $c_1$:
\begin{equation}
    \label{eq:c1-2matrices}
c_1 = p(1-p) \left[ \bar y^{\intercal}(A_0 - A_1)^2 \bar x -  \left(\bar y^{\intercal} (A_0 - A_1) \bar x\right)^2  \right].
\end{equation}
 Indeed, recalling that $Q^{-1} = Q$ (see Example~\ref{exemple:Q=matrix}) and using that  $\pi_s Q_{s,s'}= p(1-p)$ if $s\neq s'$ and $-p(1-p)$ if $s=s'$, we get
\[
c_1 = \bar y^{\intercal} \left( A_0P A_0 + A_1 P A_1 - A_0 P A_1 - A_1 P A_0 \right) \bar x,
\]
where $P = I - \bar x \bar y^{\intercal}$. Now,
\begin{align*}
    \bar y^{\intercal} A_0 \bar x \bar y^{\intercal} A_1 \bar x + & \bar y^{\intercal} A_1 \bar x \bar y^{\intercal} A_0 \bar x \\
    & - \bar y^{\intercal} A_0 \bar x \bar y^{\intercal} A_0 \bar x - \bar y^{\intercal} A_1 \bar x \bar y^{\intercal} A_1 \bar x \\
    & =  - \left(\bar y^{\intercal} (A_0 - A_1) \bar x\right)^2
\end{align*}
and
\[
\bar y^{\intercal} A_0  A_1 \bar x +  \bar y^{\intercal} A_1  A_0 \bar x 
     - \bar y^{\intercal} A_0 A_0 \bar x - \bar y^{\intercal} A_1 A_1 \bar x  =- \bar y^{\intercal}( A_0 - A_1)^2 \bar x ,
\]
which induces formula \eqref{eq:c1-2matrices}.
\end{exemple}
When considering the switching between matrices $(A_i)_{i \in \mathcal{S}}$, the following natural question arises: if all the matrices are stable (i.e. $\lambda_{\max}(A_i) < 0$ for all $i \in \mathcal{S}$), is the switched system \eqref{eq:EDOrandom} stable, in the sense that $\Lambda_{\varepsilon} < 0$? It is now known  that this is not true. In the case of random switching between two $2 \times 2$ matrices, examples of stable matrices giving an unstable system can be founded in \cite{Benaim2014Stability} and \cite{LMR14}. In the first reference, for $p=1/2$, $\lambda_{\max}(\bar A) < 0$ so that  fast  switching leads to a unstable system. In the second reference, it is a bit more complicated, since $\lambda_{\max}(\bar A) > 0$ for all $p \in (0,1)$, meaning that both slow and fast switching lead to a stable system. However, the authors prove in \cite{LMR14} that switching not too fast, nor too slowly, the system is unstable. Yet, in \cite{LMR14}, the matrices are not cooperative in the sense of Assumption~\ref{assu:Lyapunov}. This is not surprising, since it is proven in \cite{gurvits07} that switching between cooperative matrices of size $2 \times 2$ such that every matrices in the convex hull of the given matrices are stable, will \textit{always} lead to a stable system. However, it is also shown in \cite{gurvits07} that it is possible  in some  higher dimension  to construct  an example where all the matrices in the convex hull are stable, and for which there exists a periodic switching such that the linear system explodes. Later, an explicit example in dimension 3 was given by Fainshil, Margaliot and Chiganski \cite{onpls}. Precisely, consider the matrices
\begin{equation}\label{eq:A0A1}
    A_0 = \begin{pmatrix}
   -1 & 0 & 0\\
   10 & -1 & 0\\
   0 & 0 & -10
\end{pmatrix}, \; A_1 = \begin{pmatrix}
   -10 & 0 & 10\\
   0 & -10 & 0\\
   0 & 10 & -1
\end{pmatrix}.
\end{equation}
It is shown in \cite{onpls} that every convex combination of  $A_0$ and $A_1$ is stable, and yet a switch of period 1 between  $A_0$ and $A_1$ yields an explosion. In \cite{benaim2019}, the authors asked the question whether the same system but with a Markovian switching can lead to explosion. Using numerical simulations, they suggest that this is true and the Lyapunov exponent is positive for not too fast nor too slow switching. Now, using Formula~\eqref{eq:c1-2matrices}, we prove rigorously the following assertion:

\begin{prop}
There exists two cooperative matrices $B_0$, $B_1$ such that:
\begin{enumerate}
    \item For all $p \in [0,1]$, $\lambda_{\max}((1-p)B_0 + p B_1) < 0$,
    \item For some $\varepsilon > 0$ and $p \in (0,1)$, the top Lyapunov exponent $\Lambda_{\varepsilon}$ of the system~\eqref{eq:EDOrandom} with $\sigma$ generated by $Q$ given in~\eqref{eq:ratematrix2x2} is positive.
\end{enumerate}
\end{prop}
 \begin{proof}
 To emphasize the dependency on $p$,we write $\bar A(p), \Lambda_{\varepsilon}(p)$ and $c_1(p)$ for the mean matrix, the top Lyapunov exponent and the first order derivative, respectively, when the switching rates are given by the matrix $Q$ in~\eqref{eq:ratematrix2x2}, and the matrices are \eqref{eq:A0A1}. We also write $\lambda_{\max}(p)$ for $\lambda_{\max}(\bar A (p))$. In Figure~\ref{fig:lambdamax} is plotted $p \mapsto \lambda_{\max}(p)$. One sees that the maximal value of this function is attained at some $p^* \in [0.3, 0.5]$ and worth $\lambda_{\max}^* \in [-0.5, -0.45]$. One sees on Figure~\ref{fig:c1}, where is plotted $p \mapsto c_1(p)$, that on the interval $[0.3,  0.5]$, we have $c_1(p) \geq 15$, so that in particular $c_1(p^*) \geq 15 > 0$. 

\begin{figure}
\begin{subfigure}{.5\textwidth}
  \centering
  \includegraphics[width=.9\linewidth]{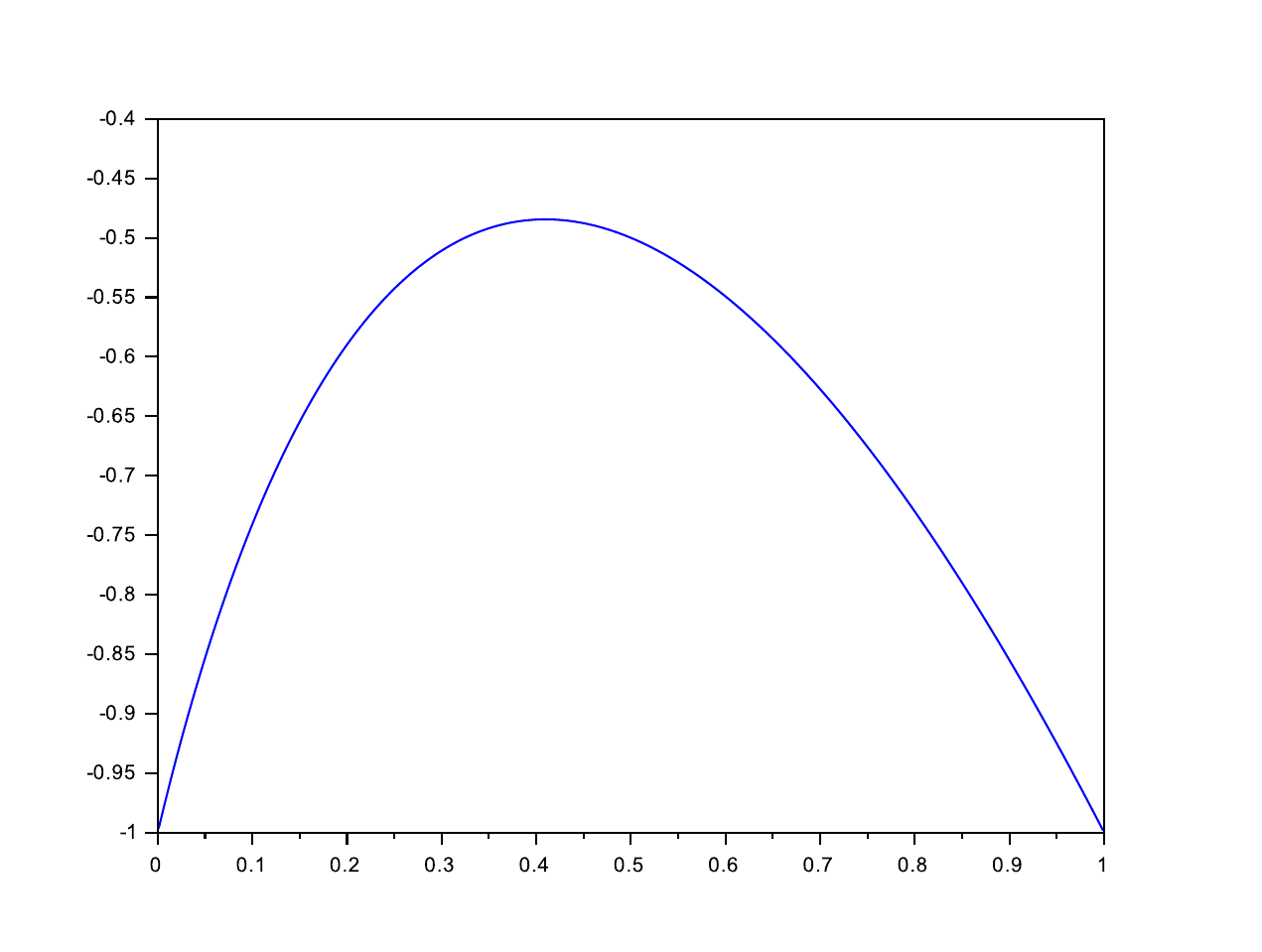}
  \caption{$p \mapsto \lambda_{\max}(p)$}
  \label{fig:lambdamax}
\end{subfigure}% 
\begin{subfigure}{.5\textwidth}
  \centering
  \includegraphics[width=.9\linewidth]{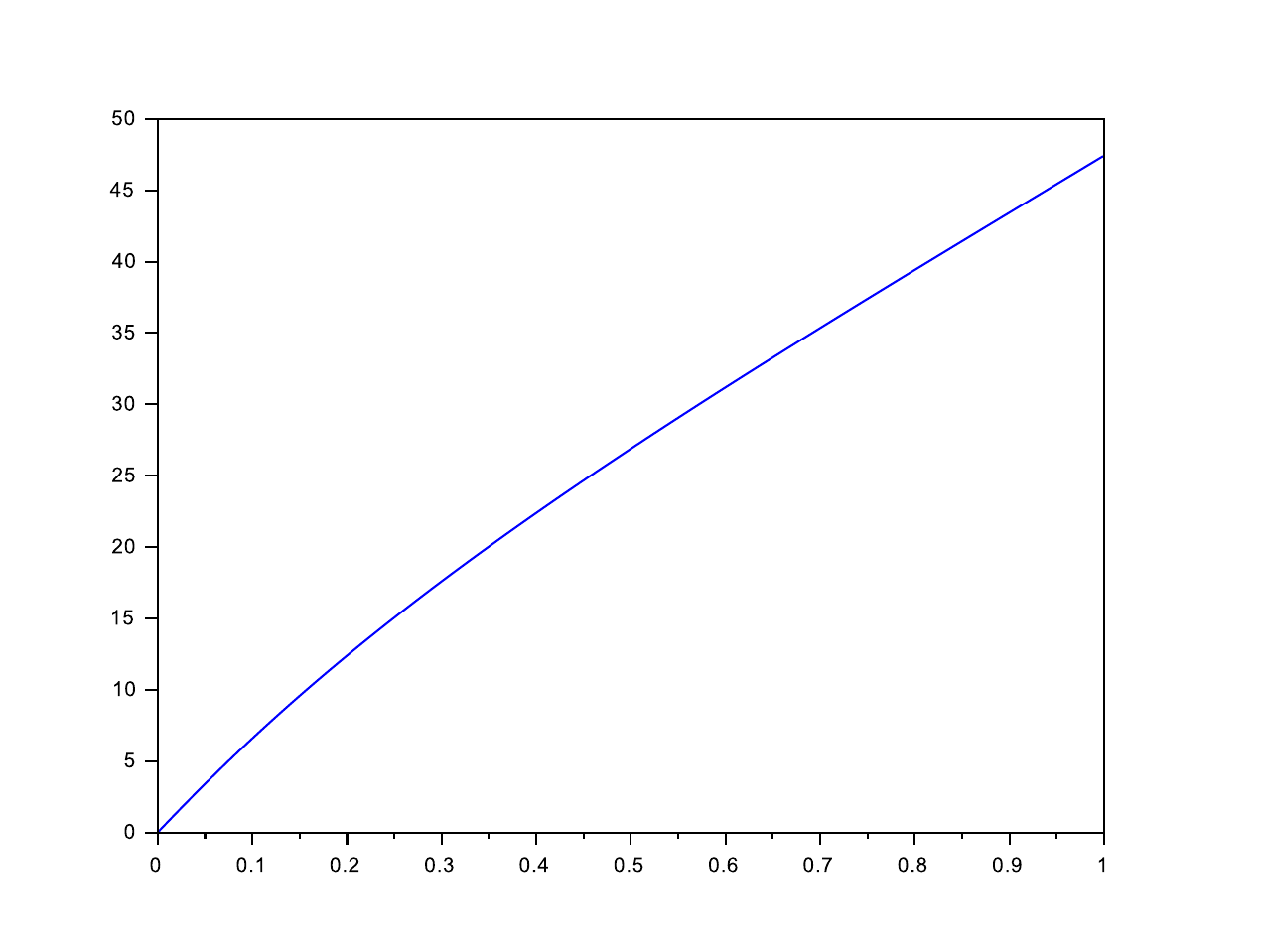}
  \caption{$p \mapsto c_1(p)$}
  \label{fig:c1}
\end{subfigure}
\end{figure} 

Let $\delta > 0$ and $A_i^{\delta} = A_i + \delta I$. Replacing $A_i$ by $A_i^{\delta}$ in~\eqref{eq:EDOrandom}, we easily check that the Lyapunov exponent of the system is $\Lambda_{\varepsilon}^{\delta}(p) = \Lambda_{\varepsilon}(p) + \delta$. Moreover, $f_i^{\delta}(x) = \1 \cdot (A_i + \delta I)x$, so that $F_i^{\delta} = F_i$, and the dynamic of the angular part of~\eqref{eq:EDOrandom} is independent of $\delta$.
Hence, Proposition~\ref{prop:lyapunov} entails
\[
\Lambda_{\varepsilon}^{\delta}(p) \geq  \lambda_{\max}(p) + \delta + c_1(p) \varepsilon - M(p) \varepsilon^2,
\]
for some constant $M(p)$ depending on $p$. Choosing $p=p^*$ and $\varepsilon = \frac{c_1(p^*)}{2M(p^*)} > 0$ yields
\[
\Lambda_{\varepsilon}^{\delta}(p^*) \geq  \lambda_{\max}( p^*) + \delta + \frac{c_1(p^*)^2}{4M(p^*)}.
\]
Letting $\delta = - \frac{c_1(p^*)^2}{8M(p^*)} - \lambda_{\max}(p^*)$ concludes the proof, since
\begin{enumerate}
    \item $\Lambda_{\varepsilon}^{\delta}(p^*) \geq \frac{c_1^2(p^*)}{8M(p^*)} > 0;$
    \item For all $p \in [0,1]$, $\lambda_{\max}(A_p^{\delta}) \leq \lambda_{\max}(A_{p^*}^{\delta}) = \lambda_{\max}(p^*) + \delta = - \frac{c_1^2(p^*)}{8M(p^*)} < 0.$
\end{enumerate}
  \end{proof}
  
  \begin{exemple}
Assume as in Example~\ref{exemple:Q=pi-I} that $Qf = Q^{-1}f = \pi f - f$. Then, $Q^{-1}(A)(s) = \bar A - A(s)$. Using $\bar y^{\intercal} \bar A =\lambda( \bar A)\bar y^{\intercal} $ and $\bar y^{\intercal}  \bar x= 1$, we get that $\bar y^{\intercal} \bar A ( \bar x \bar y^{\intercal}  - I) = 0$, which yields
\[
c_1 = \int_{\mathcal{S}} \bar y^{\intercal} A^2(s) \bar x - \left( \bar y^T A(s) \bar x \right)^2 \pi(ds)\,.
\]
Notice that, from Jensen's inequality,
\[
\int_{\mathcal{S}}  \left( \bar y^T A(s) \bar x \right)^2 \pi(ds) \geq \left( \int_{\mathcal{S}}  \bar y^T A(s) \bar x \pi(ds) \right)^2 = \lambda_{\max}(\bar A)^2\,,
\] 
so that
\[
    c_1  \leq  \int_{\mathcal{S}} \bar y^{\intercal} A^2(s) \bar x \pi(ds) - \lambda_{\max}(\bar A)^2  = \bar y^{\intercal} \overline{A^2} \bar x - \lambda_{\max}(\overline{A}^2) = \bar y^{\intercal} \left( \overline{A^2} - \overline{A}^2 \right) \bar x\,.
\]
  \end{exemple}

%Then for $p= 1/2$,
%\[
%\bar A=\frac{1}{2}\begin{pmatrix}
%-11  & 0 & 10\\
%10 & - 11 & 0\\
%0 & 10 & - 11
%\end{pmatrix}
%\]
%Since the sum of the entries of a row or a column is always $-1/2$, the principal eigenvalue of $\bar A$ is $-1/2$, with associated right eigenvector $\bar x = \frac{1}{3} \1$ and associated left eigenvector $\bar y = \1$. Using formula \eqref{eq:c1-2matrices}, easy computation then leads to
%\[
%c_1 = \frac{1}{4} \left( \frac{323}{3} - \frac{1}{9} \right) > 0.
%\]

\subsection{Richardson extrapolation for randomized splitting schemes}

In \cite{RandomSplitting}, in order to approximate the solution of an ODE of the form
\begin{equation}\label{eq:ODEMattingly}
    \dot y(t) = \bar F(y(t))\,,
\end{equation}
the vector field $\bar F$ is splited as  $\bar F(x) = \sum_{k=1}^N F_k(x)$, where each flow $\dot x=F_k(x)$ can be solved exactly. More precisely, the two main examples of interest in \cite{RandomSplitting} (to which we refer for details and motivations) are the so-called Lorenz-96 model and a finite-dimensional  Galerkin projection of the   vorticity formulation of 2D Navier-Stokes.  The trajectory $y$ is then approximated by a discrete time scheme of the form
\[x^\varepsilon_{n+1} = \Phi^1_{\tau_1/\varepsilon}\circ \dots \circ \Phi^N_{\tau_N/\varepsilon} (x^\varepsilon_n)\]
where $\Phi^i_t$ is the flow associated to $F_i$ at time $t$ and $\tau_1,\dots,\tau_N$ are independent random variables distributed according to the standard exponential law. More precisely, for a fixed $t>0$, $y(t)$ is approximated by $x_n^\varepsilon$ with $n\varepsilon =t$, for a small $\varepsilon$. Notice that this does not enter directly our framework. However, consider the Markov-modulated ODE
\[ \dot x_\varepsilon(t) = F_{\sigma(N t/\varepsilon)} (x_\varepsilon(t))\,,\] 
where $\sigma$ is a cyclic chain on $\cco 1,N\ccf$ with rate $1$ (i.e. it is a standard Poisson Process modulo $N$). We can take the times $\tau_i $ to define the jump times of $\sigma$, and thus $x_{n}^\varepsilon$ is exactly $x_\varepsilon(\varepsilon S_n/N)$ where $S_n$ is the $(nN)^{th}$ jump time of $\sigma$. In particular, if $n\varepsilon = t$, assuming that the vector fields $F_k$ are bounded,
\[|x_\varepsilon(t) - x_n^\varepsilon| = |x_\varepsilon(t) - x_\varepsilon\po \varepsilon S_n /N \pf | \leqslant \max_{k\in\cco 1,N\ccf}\|F_k\|_\infty |\varepsilon S_n/N-t|\,,\]
and
\[\mathbb E\left|\frac{\varepsilon S_n}N-t\right|^2 = \frac{n \varepsilon^2}{N} =  \frac{t \varepsilon}{N}\,. \]
In other words, $(x_\varepsilon(t))_{t\geqslant 0}$ (which can be sampled exactly if $(x_n^\varepsilon)_{n\in\N} $ can) can be seen as an alternative variation of the scheme of \cite{RandomSplitting}.

The convergence of $x_n^\varepsilon$ to $y(t)$ as $\varepsilon \rightarrow 0$ for a fixed $t=n\varepsilon$, or more precisely of the corresponding Markov transition operators, is stated in \cite[Theorem 4.1]{RandomSplitting}, which is thus similar to our Proposition~\ref{prop:dev-semigroup} at order $n=0$. Notice that, a priori, the ODE \eqref{eq:ODEMattingly} is in $\R^d$ but \cite[Theorem 4.1]{RandomSplitting} is proven by restricting the study on the orbit of the process, which is assumed to be bounded (this is Assumption 1 of \cite{RandomSplitting}), so that our results apply in this context.

The higher order expansion in Proposition~\ref{prop:dev-semigroup} enables the use of a Richardson extrapolation to get better convergence rates: for instance we obtain, for a fixed $t>0$ and $f\in\mathcal A$,
\[\left| f\po \varphi_t(x)\pf - 2 \mathbb E \po f(x_{\varepsilon/2}(t))\pf  + \mathbb E \po f(x_{\varepsilon}(t))\pf  \right| \leqslant C\|f\|_3 \varepsilon^2\,, \]
for some constant $C>0$.

\subsection{Species invasion rate in a two species Lotka-Volterra model in random environment}
\label{subsec:LV}
We consider a model of two species in interaction, in a random environment, described by the following system of Lotka - Volterra equations:
\begin{equation}
\begin{cases}
\frac{d X(t)}{dt} =  X(t) \left(a_ {10}^{\sigma(t)} - a_{11}^{\sigma(t)} X(t) + a_{12}^{\sigma(t)} Y(t) \right)\\
\frac{d Y(t)}{dt} =  Y(t) \left( a_{20}^{\sigma(t)} + a_{21}^{\sigma(t)} X(t) - a_{22}^{\sigma(t)} Y(t) \right)
\end{cases}
\label{eq:LV2dgeneral}
\end{equation}
We assume that $a^s_{11}, a^s_{22} > 0$ (intraspecific competition), the other parameters can be positive or negative. The coefficients $a_{10}^s$ and $a_{20}^s$ represents the growth rates of species 1 and 2 respectively, while the coefficients $a_{12}^s$ and $a_{21}^s$ models the interaction between the two species. If for example $a_{10}^s, a_{21}^s >0$ while $a_{20}^s, a_{12}^s < 0$, we get a prey - predator system (species 1 being the prey, and species 2 the predator). Assumption~\ref{assu:settings} is also enforced in this whole section.

In the absence of species 2 and after accelerating the dynamics of $\sigma$ by a factor $1/\varepsilon$, species 1 evolves according to a logistic equation in random environment:
\begin{equation}\label{eq:Lotka1}
\frac{d X(t)}{dt} = X(t) \left( a_{10}^{\sigma(t/\varepsilon)}  - a_{11}^{\sigma(t/\varepsilon)} X(t) \right)\,.
\end{equation}
We prove in the Annex the following result:
\begin{lem}
\label{lem:uniqueIPM}
Assume that  $p_0:= \inf_{s \in \mathcal{S}} \frac{a_{10}^s}{a_{11}^s} > 0$, $p_1:=\sup_{ s\in \mathcal{S}} \frac{a_{10}^s}{a_{11}^s} < + \infty$ and $\inf_{s \in \mathcal{S}} a_s^{11} > 0$.
Then the process $(X(t), \sigma(t/\varepsilon))_{t\geqslant 0}$ with $X$ given by \eqref{eq:Lotka1} and $X(0) \neq 0$ will ultimately lie in $E:= [p_0, p_1] \times \mathcal{S}$ and  admits a unique stationary distribution $\mu_{\varepsilon}$ on $E$.  Moreover, $(X(t), \sigma(t/ \varepsilon))$ converges in law to $\mu_{\varepsilon}$ when $t$ goes to infinity.
\end{lem}
The stationary distribution $\mu_{\varepsilon}$ represents the population of species 1 at equilibrium in the absence of species 2. The invasion growth rate of species 2 when species 1 is at equilibrium is then given by
\begin{equation}
\label{eq:Lambday}
\Lambda_y(\varepsilon) = \int_{ \mathcal{M} \times \mathcal{S} } (a_{20}^s  + a_{21}^s x) \mu_{\varepsilon} (dx, ds)\,.
\end{equation}
Intuitively, we are interested in the sign of this quantity, since when the abundance of the second species is close to $0$, we have roughly
\[
\frac{1}{t}\ln( Y(t)) \simeq \frac{1}{t} \int_0^t \co  a^{\sigma(u/\varepsilon)}_{20} + a^{\sigma(u/\varepsilon)}_{21} X(u)\cf  du,
\]
where $X$ solves~\eqref{eq:Lotka1} and the right hand side converges to $\Lambda_y(\varepsilon)$ when $t$ goes to infinity, thanks to Lemma~\ref{lem:uniqueIPM} and the ergodic theorem. It is made rigorous in \cite{B18}  that, indeed, the sign of $\Lambda_y(\varepsilon)$ gives information on the local behaviour of~\eqref{eq:LV2dgeneral} near the boundary $\{y = 0\}$. Denoting by $\bar \alpha = \int_{\mathcal S} \alpha^s \pi(ds)$ the mean with respect to $\pi$ of a function $\alpha$ defined on $\mathcal{S}$,  it is proven in \cite{BL16}, in the specific case when $\sigma$ is a two-states Markov chain, that $\Lambda_y(\varepsilon)$ converges as $\varepsilon$ goes to 0 to $\bar a_{20} + \bar a_{21} \bar x$. The proposition below extends this result to the general case of a process $\sigma$ satisfying  Assumption~\ref{assu:settings} and with a first order expansion:
%Notice that, if $a_s^{10}$ is not bounded below by a positive constant (and in particular, if it can take negative values) or if $a_s^{11}$ is not bounded, then $0$ is accessible for the single species process $X$ solution to 

%%\frac{d X(t)}{dt} = X(t) \left( a_{10}^{\sigma(t/\varepsilon)}  - a_{11}^{\sigma(t/\varepsilon)} X(t) \right)
%\end{equation}
%%and thus $\mu^0 = \delta_0 \otimes \pi$ is a stationary distribution of $(X, \sigma)$. However, we are interested in a potential ergodic measure $\mu_{\varepsilon}$ which does not gives mass to $\{0\} \times \mathcal{S}$. Such a measure exists if and only if $\bar{a}_{10}> 0$. But, even though in that case $\bar F$ admits a unique positive equilibrium $\bar x$, Assumption~\ref{assu:fast} is not satisfied since $\mathcal{M}$ has to contain $0$. Therefore, we have to assume that $\inf_{s \in \mathcal{S}} a_{10}^s > 0$ and $\sup_{ s\in \mathcal{S}} a_{11}^s < + \infty$. 

\begin{prop}
Assume that  $\inf_{s \in \mathcal{S}} a_{10}^s > 0$ and $\sup_{ s\in \mathcal{S}} a_{11}^s < + \infty$. Then, for any invariant probability measure $\mu_{\varepsilon}$ of $(X_t, \sigma(t/\varepsilon))_{t \geq 0}$,
\[
\Lambda_y(\varepsilon) = \bar{a}_{20} + \bar{a}_{21}\bar x + \varepsilon \frac{\bar{a}_{21} \bar{a}_{10}^2}{\bar{a}_{11}}\int_{\mathcal{S}} \left( \frac{a_{10}^s}{\bar{a}_{10}} - \frac{a_{11}^s}{\bar{a}_{11}} \right) Q^{-1} \left(  \frac{a_{11}}{\bar{a}_{11}} - \frac{a_{21}}{\bar{a}_{21}}  \right)(s) \pi(ds) + \underset{\varepsilon\rightarrow 0}{o}(\varepsilon).
\]
\end{prop}

\begin{proof}
 
First, note that the unique equilibrium of $\bar F(x) = x(\bar{a}_{10} - \bar{a}_{11} x)$ is $\bar x := \bar{a}_{10}/\bar{a}_{11}$ and Assumption~\ref{assu:fast} holds (since thanks to Lemma~\ref{lem:uniqueIPM}, it is sufficient to consider initial conditions in the compact set $ [p_0, p_1]$).
By Theorem~\ref{thm:expansion_mu},  and since the marginal of $\mu_\varepsilon$ in $s$ is independent from $\varepsilon$ (so that $\mu_\varepsilon(a_{20}) = \bar{a}_{20}$), one has
\begin{equation}
\Lambda_y(\varepsilon) = \bar{a}_{20}   + \bar{a}_{21} \bar{x} + \varepsilon c_1 + \underset{\varepsilon\rightarrow 0}{o}(\varepsilon),
\end{equation}
with 
\[c_1 = \pi L_c Q^{-1} \po  L_c h - f  \pf (\bar x), \]
for $f(x,s) =  a^s_{21} x$. 
Using the formula for $h$ derived in Remark~\ref{rem:dim1}, we have
\begin{align*}
h(x,s) & = \int_{\bar x}^x \frac{\bar f(y) - \bar f(\bar x)}{\bar F(y)}dy\\
& = \int_{\bar x}^x   \bar{a}_{21} \frac{y- \bar x }{ y ( \bar{a}_{10} - \bar{a}_{11} y)} dy\\
& = \frac{\bar{a}_{21}  }{\bar{a}_{11}} \ln\left( \frac{\bar x}{x} \right),
\end{align*}
which yields
\[
L_c h(x,s) = -(a_{10}^s - a_{11}^s  x) \frac{  \bar{a}_{21}  }{\bar{a}_{11}}.
\]
Then, $\na_x(L_c h - f) = a_{11} \bar a_{21}/\bar a_{11} - a_{21} $, so that
\begin{align*}
    L_c Q^{-1} \po  L_c h -      f  \pf (\bar x,s) &= \bar x (a_{10}^s - a_{11}^s \bar x)Q^{-1}\po \frac{  \bar{a}_{21}  }{\bar{a}_{11}} a_{11} - a_{21}  \pf (s)  
\end{align*}
and finally, integrating this with respect to $\pi$ and using that $\bar x = \bar{a}_{10}/\bar{a}_{11}$,
\begin{equation}\label{eq:c1LotkaVolterre}
c_1 =   \frac{ \bar{a}_{21} \bar{a}_{10}^2}{\bar{a}_{11}}  \int_{\mathcal S} \po \frac{a_{10}^s}{\bar{a}_{10}} -  \frac{a^s_{11}}{\bar{a}_{11}}\pf Q^{-1}\po \frac{   a_{11}}{\bar{a}_{11}}    - \frac{a_{21}}{\bar{a}_{21}}  \pf (s) \pi(\dd s) \,,     
\end{equation}
which concludes.
\end{proof}

Let us discuss the sign of the expression \eqref{eq:c1LotkaVolterre} in some particular cases.

\begin{enumerate}
    \item If $a_{11}^s= \bar{a}_{11}$ and $a_{10}^s = \bar{a}_{10}$ for all $s$, then $\Lambda_{y}'(0) =0$. This is expected, since in that case, the process $X$ solution to~\eqref{eq:Lotka1} is deterministic and does not depend on $\varepsilon$, and $\mu_{\varepsilon} = \delta_{\frac{\bar{a}_{10}}{\bar{a}_{11}}} \otimes \pi$ for all $\varepsilon$. More interestingly, if $a_{11}^s = \bar{a}_{11}$ and $a_{21}^s = \bar{a}_{21}$ but $a_{10}$ is varying, we still have $\Lambda_y'(0) = 0$.
    \item If $a_{10}^s = \bar{a}_{10}$ and $a_{21}^s = \bar{a}_{21}$ for all $s$, using that $\pi Q^{-1} = Q^{-1} 1 = 0$, we get
    \begin{eqnarray*}
    c_1 &=&   - \frac{ \bar{a}_{21} \bar{a}_{10}^2}{\bar{a}_{11}}  \int_{\mathcal S}  \frac{a_{11}^s}{\bar{a}_{11}} Q^{-1}\po\frac{a_{11}}{\bar{a}_{11}}   \pf (s) \pi(\dd s) \\
    &=& - \frac{ \bar{a}_{21} \bar{a}_{10}^2}{\bar{a}_{11}^3}  \int_{\mathcal S}  (a_{11}^s -\bar{a}_{11}) Q^{-1}\po a_{11} - \bar{a}_{11}   \pf (s) \pi(\dd s)  \,. 
    \end{eqnarray*}
This is always non-negative, as it can be interpreted as an asymptotic variance. Indeed, considering $g\in\mathcal A$ with $\pi g=0$, for $t\geqslant0$,
\begin{eqnarray*}
t \mathbb E_{\pi} \left|\frac1t\int_0^t g(\sigma(s))\dd s \right|^2 & = & \frac2t \mathbb E_{\pi} \int_0^t \int_s^t g(\sigma(u))g(\sigma(s)) \dd s \dd u  \\
& = & \frac2t   \int_0^t \int_s^t \pi \po  g \mathcal Q_{u-s}g \pf \dd s \dd u \\
& =&  \frac2t   \int_0^t \int_0^t \1_{v<t-s}\pi \po  g \mathcal Q_{v}g \pf \dd s \dd v \\
& =&  \frac2t   \int_0^t   (t-v)\pi \po  g \mathcal Q_{v}g \pf   \dd v \\
&\underset{t\rightarrow \infty}\longrightarrow& - 2 \pi \po g Q^{-1} g\pf\,, 
\end{eqnarray*}
where we used that $\int_0^t \mathcal Q_v g $ converges to $-Q^{-1} g$ and bounded
\[\left|\int_0^t   v \pi \po  g \mathcal Q_{v}g \pf   \dd v\right| \leqslant  \int_0^{\infty} Cv e^{-\gamma v}\|g\|_\infty^2 \dd v \,.\]
\end{enumerate}

\begin{exemple}
We consider the two-state case, i.e.,  $\sigma$ is Markov chain on $\mathcal{S} = \{0,1\}$, with matrices rates given by 
\[
Q =  \begin{pmatrix}
- p & p \\
1 -p & -(1 - p)
\end{pmatrix}
\]
for some $p \in (0,1)$. The behaviour of the solutions to \eqref{eq:LV2dgeneral} was studied by Benaïm and Lobry in \cite{BL16}, through the signs of the invasion growth rates of species 1 and 2, in the competitive case (ie, when $a_{21}^s$ and $a_{12}^s$ are negative). Their study was complemented by Malrieu and Zitt in \cite{MZ17}, who proposed an alternative formula for the invasion growth rate which made possible for them to understand the monotonicity of $\varepsilon \mapsto \Lambda_y( \varepsilon)$. In particular, it is a consequence of the proof of Lemma 4.1 in \cite{MZ17} that $\varepsilon \mapsto \Lambda_y(\varepsilon)$ is increasing if $A_2 > 0$ while it is decreasing if $A_2 < 0$, where $A_2$ is the coefficient of the second order term of the polynomial 
\[
P(x) = \left[ \frac{a_{20}^1}{a_{10}^1}(1 +  \frac{a_{21}^1}{a_{20}^1} x)( 1 - \frac{a_{11}^0}{a_{10}^0} x) - \frac{a_{20}^0}{a_{10}^0}(1 +  \frac{a_{21}^0}{a_{20}^0} x)( 1 - \frac{a_{11}^1}{a_{10}^1} x) \right] \frac{\frac{a_{11}^1}{a_{10}^1} - \frac{a_{11}^0}{a_{10}^0}}{| \frac{a_{11}^1}{a_{10}^1} - \frac{a_{11}^0}{a_{10}^0}|},
\]
that is $A_2 = (a_{11}^1 a_{21}^0 - a_{11}^0a_{21}^1)/(a_{10}^0a_{10}^1)$, if we assume without loss of generality that $\frac{a_{11}^1  }{a_{10}^1} \geqslant \frac{a_{11}^0}{a_{10}^0}$. Let us prove that our results are in accordance with the results of Malrieu and Zitt, by studying the sign of the first order term of the expansion of  $\varepsilon \mapsto \Lambda_y( \varepsilon)$ when $ \varepsilon \to 0$. Using Example~\ref{exemple:Q=matrix}, we have $\pi=(1-p,p)$, $Q^{-1}= Q$ and $\pi_s Q_{s,s'}= p(1-p)$ if $s\neq s'$ and $-p(1-p)$ if $s=s'$, from which we compute
\begin{align*}
    c_1 &= \sum_{s,s'\in{0,1}} \pi_s Q_{s,s'} \po \frac{a_{10}^s}{\bar{a}_{10}} -  \frac{a_{11}^s}{\bar{a}_{11}}\pf  \po  \frac{   a_{11}^{s'}}{\bar{a}_{11}} -  \frac{a_{21}^{s'}}{\bar{a}_{21}}   \pf \\
    & = p(1-p) \Big[ \po \frac{a_{10}^0}{\bar{a}_{10}} -  \frac{a_{11}^0}{\bar{a}_{11}}\pf  \po   \frac{a_{11}^1}{\bar{a}_{11}} - \frac{   a_{21}^1}{\bar{a}_{21}} \pf
    + \po \frac{a_{10}^1}{\bar{a}_{10}} -  \frac{a_{11}^1}{\bar{a}_{11}}\pf  \po \frac{a_{11}^{0}}{\bar{a}_{11}} - \frac{   a_{21}^{0}}{\bar{a}_{21}}   \pf \\
    & \qquad - \po \frac{a_{10}^0}{\bar{a}_{10}} -  \frac{a_{11}^0}{\bar{a}_{11}}\pf  \po  \frac{a_{11}^{0}}{\bar{a}_{11}} - \frac{   a_{21}^{0}}{\bar{a}_{21}} \pf - \po \frac{a_{10}^1}{\bar{a}_{10}} -  \frac{a_{11}^1}{\bar{a}_{11}}\pf  \po  \frac{a_{11}^{1}}{\bar{a}_{11}} - \frac{   a_{21}^{1}}{\bar{a}_{21}}  \pf\Big]\\
    & = p(1-p)  \po \frac{a_{10}^0}{\bar{a}_{10}} -  \frac{a_{11}^0}{\bar{a}_{11}} - \po\frac{a_{10}^1}{\bar{a}_{10}} -  \frac{a_{11}^1}{\bar{a}_{11}}\pf \pf     \po  \frac{a_{11}^{1}}{\bar{a}_{11}} - \frac{   a_{21}^{1}}{\bar{a}_{21}}  - \po \frac{a_{11}^{0}}{\bar{a}_{11}} - \frac{   a_{21}^{0}}{\bar{a}_{21}}  \pf \pf\\
    & = p(1-p)  \po \frac{a_{10}^0-a_{10}^1}{\bar{a}_{10}} -  \frac{a_{11}^0-a_{11}^1}{\bar{a}_{11}} \pf     \po \frac{a_{11}^1-a_{11}^0}{\bar{a}_{11}} -  \frac{a_{21}^1-a_{21}^0}{\bar{a}_{21}} \pf\\
    & = \frac{p(1-p)a_{10}^0 a_{10}^1 }{\bar{a}_{10} (\bar{a}_{11})^2 \bar{a}_{21} } \left( \frac{a_{11}^1}{a_{10}^1} - \frac{a_{11}^0}{a_{10}^0} \right) \left(a_{11}^0 a_{21}^1 - a_{11}^1 a_{21}^0\right) \,.
\end{align*}
Since we have assumed that $\frac{a_{11}^1  }{a_{10}^1} \geqslant \frac{a_{11}^0}{a_{10}^0}$,  the sign of $\Lambda_y'(0)$ is indeed  the same as that of $a_{11}^1 a_{21}^0 - a_{11}^0 a_{21}^1$ (recall that $\bar{a}_{21} < 0$). 
\end{exemple}

\begin{exemple}
We now consider the case described in Example~\ref{exemple:Q=pi-I}, where $Qf = Q^{-1} f = \pi f - f$. In that case, we get 
\[
Q^{-1} \left( \frac{a_{11}}{\bar{a}_{11}} -  \frac{   a_{21}}{\bar{a}_{21}}   \right)(s) = \frac{a^s_{21}}{\bar{a}_{21}} -  \frac{   a^s_{11}}{\bar{a}_{11}}   
\]
which entails
\[
    c_1 = \frac{\overline{a_{10} a_{21}}}{\bar{a}_{10} \bar{a}_{21}} + \frac{\overline{a_{11}^2}}{(\bar{a}_{11})^2} - \frac{\overline{a_{10} a_{11}}}{\bar{a}_{10} \bar{a}_{11}} -\frac{\overline{a_{11} a_{21}}}{\bar{a}_{11} \bar{a}_{21}} \,.
\]
\end{exemple}

 \section{Annex}
 \subsection{Proof of Proposition~\ref{prop:dev-semigroup}}
 
 \begin{proof}
 This is essentially the proof of \cite{PTW2012}, in dual form since we work with $P_t f$ rather than the law of the process. It is organized in three steps.  First, assuming formally that \eqref{eq:dev-pt} gives an expansion in $\varepsilon$ of $P_t^\varepsilon$, we deduce the expression of $P_t^{(k)}$ and $S_t^{(k)}$. With these definitions, in a second step, we establish  the bounds \eqref{eq:boundsP} and \eqref{eq:boundsS} on $P_t^{(k)}$ and $S_t^{(k)}$. Finally, from this, we obtain the bound \eqref{eq:boundsR} on $R_{n,t}^\varepsilon$.

\textbf{Step 1.} 
Fix $f \in \mathcal{C}^{\infty}(\mathcal M\times \mathcal S)$.  For all $t \geq 0$,
\begin{equation}
\label{eq:kolmo}
\partial_t P_t^{\varepsilon} f = L^{\varepsilon} P_t^{\varepsilon} f.
\end{equation}
Having in mind the ansatz  that \eqref{eq:dev-pt} gives an expansion in $\varepsilon$ of $P_t^\varepsilon f$, developing in $\varepsilon$ each side of the above equality, and equating the terms of the same order (treating separately the terms $P_t^{(k)} f $ and $S_t^{(k)}f $), we end up with the following equations: for all $k \geq 0$,  
\begin{equation}
\label{eq:1/eps}
Q P_t^{(0)}f  = 0 ,
\end{equation}

\begin{equation}
\label{eq:order-k}
\partial_t P_t^{(k)}f  = L_c \left(P_t^{(k)}f\right)  + Q P_t^{(k+1)}f 
\end{equation}
and

\begin{equation}
\label{eq:layer-order-0}
\partial_t S_t^{(0)} f = Q S_t^{(0)}f,
\end{equation}

\begin{equation}
\label{eq:layer-order-k}
\partial_t S_t^{(k)} f = Q S_t^{(k)} f + L_c(S_t^{(k-1)}f).
\end{equation}
To shorten the notation, for  $k \in \N$, we write $u_k(t,x,s)=P_t^{(k)}f(x,s)$. Moreover, for a function $g(t,x,s)$, we use the notation $\pi g(t,x) = \int_{\mathcal S} g(t,x,s)\pi(\dd s)$.   Equation \eqref{eq:1/eps} is equivalent to say that there exists $\theta_{0}(t,x)$ such that for all $s\in\mathcal S$, $u_0(t,x,s) = \theta_{0}(t,x)$. Now, if we want Equation \eqref{eq:order-k} for $k = 0$ to have a solution $u_1$, this will induce a constraint, fixing the value of $\theta_{0}$. Indeed, this equation can be rewritten as \begin{equation}
\label{eq:Q-1}
Q u_{1} =    \partial_t u_{0} - L_c u_{0}.
\end{equation}
The left hand side averages to $0$ with respect to $\pi$ (for all fixed $x\in\mathcal M$ $t\geqslant 0$), and thus we have to impose that
\[
\pi \left( \partial_t  u_{0} - L_c u_{0} \right) = 0.
\]
This is equivalent to
\begin{equation}
\label{eq:transport0}
\partial_t \theta_{0}(t,x) -  \bar  F(x) \cdot\na  \theta_{0}(t,x)  = 0,
\end{equation}
which is a transport equation  with solution  given by
\[
 \theta_{0}(t,x) = \theta_{0}(0,\varphi_t(x)).
\]
In order to identify a suitable initial condition $\theta_{0}(0,x)= \pi u_0(0,x)$, we formally let $\varepsilon$ and then $t$ go to $0$ in 
 \[
\pi P_t^\varepsilon f(x)  = \pi f(x) + \int_0^t \pi L_c P_r^\varepsilon f( x) dr,
\]
which is obtained by integrating \eqref{eq:kolmo}   and using that $\pi Q = 0$. 
This yields $\theta_0(0,x) = \pi f(x)$ and as announced, 
\begin{equation}
\label{eq:formula-u0}
P_t^{(0)}f(x,s) = u_{0}(t,x,s)  = \pi f( \varphi_t(x)).
\end{equation}

Now, reasoning by induction, we assume that for some $k \geq 1$, we have constructed $u_{k-1}$ satisfying $\pi \left( \partial_t  u_{k-1} - L_c u_{k-1} \right) = 0
$, from which we want to define $u_k$. Recalling the definition \eqref{eq:Q-1} of the pseudo-inverse $Q^{-1}$, Equation \eqref{eq:order-k} is then equivalent to the existence of a function $\theta_{k} : \R_+ \times \R^d \to \R$ such that
\begin{equation}
\label{eq:uk}
u_{k}(t,x,s) = Q^{-1}( \partial_t u_{k-1} - L_c u_{k-1})(t,x,s)  + \theta_{k}(t,x)   ,
\end{equation}
Once again, to find a suitable $\theta_{k}$, we use Equation \eqref{eq:order-k} at order $k+1$ that we rewrite as
\begin{equation}
\label{eq:Q-2}
Q u_{k+1} = \partial_t u_k -  L_c u_k.
\end{equation}
As before, for fixed $(t,x)$, since the left-hand side averages to $0$ with respect to $\pi$, we impose
\begin{equation}
\label{eq:pideltau}
 \pi \left( \partial_t  u_k - L_c u_k \right)(t,x) = 0.
\end{equation}
Notice that $ \pi Q^{-1} = 0$, so that, averaging Equation \eqref{eq:uk} with respect to $\pi$ and differentiating with respect to time, 
\begin{equation}
\label{eq:pideltau2}
\pi \partial_t  u_k  = \partial_t \theta_{k}.
\end{equation}
Injecting Equations \eqref{eq:pideltau2} and \eqref{eq:uk} in \eqref{eq:pideltau} yields
\begin{equation}
\label{eq:transport1}
\partial_t \theta_{k}(t,x)  - \bar  F(x) \cdot \nabla \theta_{k}(t,x)  =  \pi
L_c \bb_k(t,x),
\end{equation}
where we set 
\[
\bb_k = Q^{-1} ( \partial_t u_{k-1} - L_c u_{k-1}).
\]
%and
%\begin{align}
%b_t(x,i) & = (\bb_t(x))_i\\
%& = \sum_j Q^g_{i,j} [  \na \bar f(\varphi_t(x)) \cdot \bar F( \varphi_t(x)) -  F_j(x) \cdot D\varphi_t(x) \na \bar f( \varphi_t(x)) ].
%\end{align}
This is again a transport equation, similar to \eqref{eq:transport0}, but with a source term. The solution is given by
\begin{equation} \label{eq:thetaktheta0}
\theta_k(t,x) = \theta_k(0,\varphi_t(x)) + \int_0^t  \pi  L_c \bb_k(r,\varphi_{t-r}(x)) dr.
\end{equation}
For now, by construction, for any choice of the function $\theta_k(0,\cdot)$, the function $u_k$  defined by \eqref{eq:uk} where $\theta_k$ is given by \eqref{eq:thetaktheta0} satisfies \eqref{eq:order-k} and \eqref{eq:pideltau}. It remains to specify a suitable initial condition $\theta_k(0,x)$. To do so, we need first to look at the so-called boundary layer correction terms $S^{(j)}$. 

Set $v_k(t,x,s)=S_t^{(k)}f(x,s)$. Then, Equations~\eqref{eq:layer-order-0} and \eqref{eq:layer-order-k} read
\begin{equation}
\label{eq:v-order-0}
\partial_t v_0 = Q v_0
\end{equation}
and
\begin{equation}
\label{eq:v-order-1}
\partial_t v_k = Q v_k + L_c v_{k-1}.
\end{equation}
Equation \eqref{eq:v-order-0} is simply solved as  
\[
 v_0(t,x,s) = \mathcal Q_t  v_0(0,x,s) .
\]
Moreover, taking $t = 0$ and letting $\varepsilon \to 0$ in \eqref{eq:dev-pt} yields $v_0(0,x,s) + u_0(0,x,s) = f(x,s)$, hence we set
\begin{equation}
\label{eq:formula-v0}
v_0(t,x,s) = \mathcal Q_t   f(x,s) -   \pi f(x)  ,
\end{equation}
which indeed solves \eqref{eq:v-order-0}. Next, assuming by induction that $v_{k-1}$ has been defined for some $k \geq 1$, we can solve equation \eqref{eq:v-order-1} to find 
\[
v_k(t,x,s) = \mathcal Q_t v_k(0,x,s) + \int_0^t \mathcal Q_{t-r} L_c v_{k-1}(r,x,s) dr.
\]
It remains to chose a suitable initial condition $ v_k(0,x)$, which will be determined by the requirement of the long-time decay  \eqref{eq:boundsS}. Indeed, since  $\lim_{t \to \infty} \mathcal Q_t=   \pi$,   \eqref{eq:boundsS} can only hold if
\[
0 = \pi\left(  v_k(0,x) + \int_0^{\infty} \mathcal Q_r L_c v_{k-1}(r,x) dr \right)
\]
and therefore,
\begin{equation}
\label{eq:formula_v_k0}
\pi v_k(0,x) = -  \int_0^{\infty} \pi \mathcal Q_r L_c v_{k-1}(r,x) dr = -  \int_0^{\infty}  \pi L_c v_{k-1}(r,x) dr.
\end{equation}
Besides, from  \eqref{eq:dev-pt} applied with $t=0$, we get that $v_k(0,x,s) + u_k(0,x,s)=0$ for all $k\geq 1$ and, integrating \eqref{eq:uk} at  $t=0$ with respect to $\pi$, $\theta_k(0,x) = \pi u_{k}(0,x) $. Hence, the only choice of $\theta_k(0,\cdot)$ that is compatible with \eqref{eq:boundsS} is
\[
\theta_k(0,x) = \int_0^{\infty}  \pi L_c v_{k-1  }(r,x) dr.
\]
Finally, $v_k(0,x,s) = - u_k(0,x,s)$ is determined using again \eqref{eq:uk} at $t=0$.
%\[v_k(0,x) = - u_k(0,x) = - Q^g ( \partial_t u_{k-1} - L_c u_{k-1})(0,x) + \theta_{k} \mathbf{1} \]
At this point, we have completely determined a definition of $P_t^{(k)}$ and $S_t^{(k)}$, which  is the one that we use in the rest on the proof. It is such that \eqref{eq:1/eps}, \eqref{eq:order-k}, \eqref{eq:layer-order-0} and \eqref{eq:layer-order-k} hold.

\textbf{Step 2} We prove that bounds \eqref{eq:boundsP} and \eqref{eq:boundsS} on $P_t^{(k)}$ and $S_t^{(k)}$  hold true.
We work by induction on $k$, starting with $k=0$. By compactness, recalling the definition \eqref{eq:formula-u0} of $u_0$, it is straightforward that for all $T > 0$, there exists $C_0(T,j ,p)$ such that for all $j   \geq 0$ and all $f \in \mathcal{C}^{\infty}(\mathcal M\times\mathcal S)$,
\[
\sup_{t \in [0,T]}\norm{ \partial_t^p u_0 }_{j }  \leq C_0(T,j ,p) \norm{f }_{j+p }.
\]  
Let $j \geq 0$, $\alpha \in \cco 1, N \ccf^j$ and $p \in \N$. Then, by Equation \eqref{eq:formula-v0}, 
\[
  \partial^{\alpha}_x v_0 =   \mathcal Q_t ( \partial^{\alpha}_x  f - \pi ( \partial^{\alpha}_x  f) ).
\]
Using  the ergodicity condition \eqref{eq:ergodiciteQ_t}, we get
 for all $f \in  \mathcal{C}^{\infty}$, $t \geq 0$ and $j \in\N$,
\[
\norm{  v_0(t, \cdot) }_{j } \leq C e^{ - \gamma  t} \norm{  f }_{j}.
\]

For $k\geq 1$, reasoning by induction, we assume that there exist constants $C_{k-1}(T,j ,p)$ and $C_{k-1}(j )$ such that 
\[
\sup_{t \in [0,T]} \norm{ \partial_t^p u_{k-1}(t,\cdot) }_{j }  \leq  C_{k-1}(T,j ,p) \norm{f}_{j+p+2(k-1)   }
\]
and for all $t \geq 0$,
\[
 \norm{  v_{k-1}(t,\cdot) }_{j }  \leq  C_{k-1}(j ) e^{ - \gamma   t} (1 + t^{k-1}) \norm{f}_{j+2(k-1)} .
\]
Since $s,x\mapsto F_{s}(x)$ and its derivatives in $x$ are uniformly bounded over $\mathcal M\times\mathcal S$, it is straightforward to check that that for any $j, p  \in \N$, and $T > 0$, there exist $C(j )$ and $C(T,j,p)$ such that for any $g \in \mathcal{C}^{\infty}(\mathcal M)$,
\begin{eqnarray*}
\norm{ L_c g }_{j } & \leq & C(j  )  \norm{ g }_{j+1  } \\
\sup_{t \in [0,T]} \norm{\partial_t^p ( g \circ \varphi_t) }_{j } & \leq & C(T,j,p) \norm{g}_{j+p },
\end{eqnarray*}
which we use extensively in the rest of this step. In the following, $T,p,j ,k$ are fixed, and we denote by $C$ a constant which  may change from line to line and depends only on $T,p,j ,k$. First, for all $t \leq T$
\begin{align*}
\norm{ \partial_t^p (\theta_k(0, \cdot) \circ \varphi_t) }_{j } & \leq C  \norm{ \theta_k(0, \cdot) }_{j+p}  \\
& \leq C  \int_0^{\infty} \norm{L_c v_{k-1}(r, \cdot)}_{j+p} dr\\
& \leq C  \int_0^{\infty} \norm{ v_{k-1}(r, \cdot)}_{j+p+1} dr\\
& \leq C  \norm{f}_{j+p+1+2(k-1)}
\end{align*}
  by induction.  Next, using that $\partial_t^p \bb_k = Q^{-1}( \partial_t^{p+1} u_{k-1} - L_c(\partial_t^p u_{k-1}))$, we get that for all $t \leq T$ 
\begin{align*}
\norm{ \partial_t^p \bb_k }_{j }  & \leq C \left( \norm{\partial_t^{p+1} u_{k-1}}_{j } + \norm{L_c(\partial_t^p u_{k-1})}_{j } \right)\\
&\leq  C\left( \norm{\partial_t^{p+1} u_{k-1}}_{j } +  \norm{\partial_t^p u_{k-1}}_{j+1 }\right)\\
& \leq C  \norm{f}_{j+p+1+2(k-1)   }\,,
\end{align*}
  by induction. Now, from  \eqref{eq:thetaktheta0}, for all $t\leq T$,
\begin{align*}
\norm{\theta_k(t, \cdot)  }_{j } &\leq  \norm{\theta_k(0,\cdot)\circ  \varphi_t}_{j} + \int_0^t \norm{L_c \bb_k(s,\cdot) \circ \varphi_{t-s}}_{j} ds \\
 &\leq  C \norm{\theta_k(0,\cdot)}_{j} + C \int_0^t \norm{ \bb_k(s,\cdot) }_{j+1} ds \\
&   \leq C \norm{f}_{j+2+2(k-1)}.
\end{align*}
To bound the derivatives in time of $ \theta_k$, instead of using  \eqref{eq:thetaktheta0}, it is simpler to work with \eqref{eq:transport1}, which gives for $p\geq 1$
\begin{align*}
 \norm{\partial_t^{p}  \theta_k(t,.) }_{j } &\leqslant \norm{\pi L_c \partial_t^{p-1} \theta_k(t, \cdot)}_{j} + \norm{\pi L_c \partial_t^{p-1} \bb_k(t, \cdot)}_{j}   \\
 &  \leq  C \norm{  \partial_t^{p-1} \theta_k(t, \cdot)}_{j+1} +C  \norm{  \partial_t^{p-1} \bb_k(t, \cdot)}_{j+1}  \\
 &  \leq  C \norm{  \partial_t^{p-1} \theta_k(t, \cdot)}_{j+1} +C  \norm{ f}_{j+p+1+2(k-1)}.
\end{align*}
Reasoning by induction on $p$ yields $ \norm{\partial_t^{p}  \theta_k(t,.) }_{j } \leq C \norm{f}_{j+p+2+2(k-1)} $. Gathering the previous bounds in \eqref{eq:uk} yields
\[\norm{\partial_t^pu_k(t,\cdot)}_{j } \leq   \norm{\partial_t^{p}  \theta_k(t,.) }_{j,0} + \norm{ \partial_t^p \bb_k(t,\cdot) }_{j } \leq C  \norm{f}_{j+p+2k  }.  \]
This concludes the study by induction of $u_k$. We now turn to the study of $ v_k$. Using Equation \eqref{eq:formula_v_k0} and omitting the dependency in $x,s$ to alleviate notations, we decompose $v_k$ as 
\begin{eqnarray*}
v_k(t) &=& \mathcal Q_t v_k(0) + \int_0^t \mathcal Q_{t-r} L_c v_{k-1}(r) dr \\ 
& = & \mathcal Q_t ( v_k(0) - \pi v_k(0)) + \int_0^t \mathcal Q_{t-r}( L_c v_{k-1}(r) - \pi L_cv_{k-1}(r) ) \\
& & - \int_t^{\infty} \pi L_c v_{k-1}(r)dr.
\end{eqnarray*}
%\[
%v_k(t) =e^{tQ} ( v_k(0) - \pi v_k(0) \mathbf{1}) + \int_0^t e^{ (t-s)Q} ( L_c v_{k-1}(s) - \pi L_cv_{k-1}(s) \mathbf{1}) - \int_t^{\infty} \pi L_c v_{k-1}(s)ds.
%\]
We have
\[
\norm{ \mathcal Q_t ( v_k(0) - \pi v_k(0)  )}_{j } \leq C  e^{ - \gamma   t} \norm{ v_k(0) }_ {j }  \leq C e^{ - \gamma   t} \norm{ f }_{j+2k }, 
\]
where we used that $v_k(0) = - u_k(0)$ and the previous bound on $u_k(0)$. Similarly,
\begin{eqnarray*}
\norm{ \int_0^t \mathcal Q_{t-r} ( L_c v_{k-1}(r) - \pi L_cv_{k-1}(r) )dr }_{j } 
&\leq & C \int_0^t e^{ - \gamma   (t-r)} \norm{ L_c v_{k-1}(r)}_{j }  dr\\
&\leq & C \int_0^t e^{ - \gamma   (t-r)} \norm{v_{k-1}(r) }_{j+1  } dr\\
& \leq & C e^{ - \gamma   t} (1+ t^k )\norm{ f }_{j+1+2(k-1) }, 
\end{eqnarray*}
where we have used the induction hypothesis on $v_{k-1}$. Finally, 
\begin{align*}
\int_t^{\infty}  \norm{ L_c v_{k-1}(r) }_{j } dr & \leq C  \int_t^{\infty}  \norm{ v_{k-1}(r) }_{j+1 } dr\\
 & \leq C  \int_t^{\infty}  (1 + r^{k-1}) e^{ - \gamma  r} dr \norm{f}_{j+1+2(k-1) }\\
 & \leq C( 1 + t^{k-1}) e^{ - \gamma   t} \norm{f}_{j+1+2(k-1)  }.
\end{align*}
This proves the desired bounds on $v_k$. In particular, at this stage, we have proven \eqref{eq:boundsP} and \eqref{eq:boundsS}.

\textbf{Step 3.} We now bound the remainder $R_{n,t}^{\varepsilon}$  given by   \eqref{eq:dev-pt} for a given $n\geqslant 1$. Introducing $\eta_t = (\partial_t - L_\varepsilon) R_{n,t}^\varepsilon f$ for a given $f\in\mathcal C^\infty$, we see that, thanks to \eqref{eq:1/eps}, \eqref{eq:order-k}, \eqref{eq:layer-order-0} and \eqref{eq:layer-order-k}, all the low order terms in $\varepsilon$ vanishes (as designed) so that
\begin{align*}
\eta_t &=   (\partial_t - L_\varepsilon)  \po P_t^{\varepsilon} f - \sum_{k=0}^n \varepsilon^k P_t^{(k)} f +  \sum_{k=0}^n \varepsilon^k S_{\frac{t}{\varepsilon}}^{(k)}f\pf  \\
&= \varepsilon^n \po L_c P_t^{(n)} f - \partial_t P_t^{(n)} f + L_c S_{t/\varepsilon}^{(n)} f\pf . 
\end{align*}
The bounds obtained in the previous step yield
 \[\sup_{t\in [0,T]} \norm{\eta_t}_{j }   \leq C \varepsilon^n \norm{f}_{j+1+2n }\]
 for some $C$ depending only on $T,j,n$. Interpreting $R_{n,t}^\varepsilon f$ as the solution of the equation $\partial_t R_{n,t}^\varepsilon f = L_\varepsilon R_{n,t}^\varepsilon f + \eta_t$ with $R_{n,0}^ \varepsilon f=0$ yields the Feynman-Kac representation
 \[R_{n,t}f(x,s) = \mathbb E_{x,s} \po  \int_0^t \eta_r(X(r),\sigma(r)) dr \pf, \]
see \cite[Theorem 6.3]{Davis}. The bounds on $\eta_t$ immediately yields
  \[\sup_{t\in [0,T]} \|R_{n,t}^\varepsilon f\|_\infty  \leq C \varepsilon^n \|f\|_{1+2n}\]
 for some $C$ depending only on $T,n$.
 
 Now,  to get similar bounds for the derivatives in space of $R_{n,t}^\varepsilon f$, we work by induction on $j$, assuming that we have already obtained that
 \begin{equation}
     \label{eq:inductionRnt}
     \sup_{t\in [0,T]} \norm{ R_{n,t}^\varepsilon f}_{j'}  \leq C \varepsilon^n \norm{ f}_{j'+1+2n}
 \end{equation}
 for all $j'<j$ for some $j\geq 1$ and some $C>0$.  Indeed, then, considering a multi-index $\alpha \in \cco 1,d\ccf^j$, differentiating the equation $\partial_t R_{n,t}^\varepsilon = L_\varepsilon R_{n,t}^\varepsilon f + \eta_t$ yields 
 \begin{equation}
 \label{eq:dtdalphaR}
 \partial_t \partial_x^\alpha R_{n,t}^\varepsilon f = L_\varepsilon \partial_x^\alpha R_{n,t}^\varepsilon f +  \partial_x^\alpha \eta_t + \sum_{|\alpha'|\leq j} B_{\alpha'} \partial_x^{\alpha'} R_{n,t}^\varepsilon f  ,    
 \end{equation}
  where, for all multi-index $\alpha'$,  $B_{\alpha'}(x,s)$ is a matrix whose coefficients are some derivatives in $x$ of $F_s$. In other words, considering a vector $G_t$ whose coefficients are all the partial derivatives in $x$ of $ R_{n,t}^\varepsilon f$ of order $j$, then $G_t$ solves an equation of the form
  \begin{equation}
      \label{eq:FeynmanKacG}
      \partial_t G_t = L_\varepsilon G_t + B' G_t + \eta_t',
  \end{equation}
 where $L_\varepsilon$ acts coefficient-wise on $G_t$, the coefficients of the matrix $B'(x,s)$ are derivatives in $x$ of $F_s$ and $\eta_t'$ gathers $\partial_x^\alpha \eta_t$ and all the terms in \eqref{eq:dtdalphaR} involving derivatives of $R_{n,t}^\varepsilon f$ of order $|\alpha'|<j$. In particular, thanks to the bounds \eqref{eq:boundsP} and \eqref{eq:boundsS} established in step 2 and to the induction assumption, we have
 \[\sup_{t\in [0,T]} \norm{\eta_t'}_{\infty}  \leq C \varepsilon^n \norm{f}_{j+1+2n}\]
 for some $C>0$ depending on $T,n,j$.  We use again the Feynman-Kac formula for $G_t$ based on Equation \eqref{eq:FeynmanKacG}, which reads
 \[G_t(x,s) = \mathbb E_{x,s} \po \int_0^t   E_r \eta_{t-r}'(X(r),\sigma(r))dr   \pf \]
 where $(E_t)_{t\geqslant 0}$ is the solution of the matrix-valued ODE $\partial_t E_t = E_tB'(X(t),\sigma(t))$ (for completeness, we provide a proof of this formula in appendix). Since $B'$ is uniformly bounded, so it $E_r$ over $r\in[0,T]$, and thus
 \[\sup_{t\in [0,T]} \norm{G_t}_{\infty}  \leq C \varepsilon^n \norm{f}_{j+1+2n}\]
  for some $C>0$ depending on $T,n,j$. Since this holds for all $\alpha \in\cco 1,d\ccf^{j}$,  this concludes the proof by induction of \eqref{eq:inductionRnt} for all $j'$. To conclude, using that
 \[R_{n,t}^\varepsilon f = R_{n+1,t}^\varepsilon f + \varepsilon^{n+1}P_t^{(n+1)} f + \varepsilon^{n+1}S_{\frac t\varepsilon}^{(n+1)} f,\]
 we get that, for all $j\geq 0$,
 \begin{align*}
     \norm{  R_{n,t}^\varepsilon f }_j &\leq  \norm{ R_{n+1,t}^\varepsilon f}_j + \varepsilon^{n+1} \norm{ P_t^{(n+1)} f}_j + \varepsilon^{n+1} \norm{S_{\frac t\varepsilon}^{(n+1)} f}_j\\
     & \leq C \varepsilon^{n+1}\norm{f}_{j+3+2n},
 \end{align*}
 where we used  \eqref{eq:boundsP}, \eqref{eq:boundsS}
 and \eqref{eq:inductionRnt} (for $n+1$). This concludes the proof of \eqref{eq:boundsR}, and thus the proof of Proposition~\ref{prop:dev-semigroup}.

\end{proof}

\subsection{Feynman Kac formula}

\begin{lem}\label{lem:FeynmanKac}
Let $L$ be  a Markov generator on a set $\mathcal S'$. For $n\geqslant 1$,
consider $f_t(z) =(f_{1,t}(z),\dots, f_{n,t}(z)) \in \R^n$  for $z\in \mathcal S'$, $t\geqslant 0$ solution of
  \begin{equation}
      \partial_t f_t(z) = L f_t(z) + A_t(z) f_t(z) + c_t(z),\quad z\in\mathcal S',\ t\geqslant 0\,,
  \end{equation}
  where $(t,z)\mapsto A$ is a bounded $n\times n$ matrix field, $(t,z)\mapsto c$ is a bounded vector field and $L  f_t$ is to be understood component-wise.   Then, for all $t\geqslant 0$ and $z\in\mathcal S'$,
  \[f_t(z) =   \mathbb E \po E_t f_{0}(Z_t) \pf + \int_0^t \mathbb E \po E_r c_{t-r}(Z_r)   \pf \dd r\,,\]
  where $(Z_t)_{t\geqslant 0}$ is a Markov process associated to $L$ with $Z_0=z$ and $(E_r)_{r\in[0,t]}$ is the solution  of the $n\times n$ matrix-valued random ODE $\partial_r E_r = E_r A_{t-r}(Z_r) $ with $E_0=I_n$, provided this is well-defined almost surely.
\end{lem}

\begin{proof}
 Fix $t>0$, $z\in\mathcal S'$. Then $(Z_r,E_r)_{r\geqslant0}$ given in the statement of the lemma is a time-inhomogeneous Markov process with generator 
  \[\mathcal L_r g(z,e) = e A_{t-r}(z)\cdot \na_e g(z,e) + L g(z,e), \]
  and thus, considering the vector-valued function $h(r,z,e)=e f_{t-r}(z)$, we end up with
\begin{eqnarray*}
\partial_r \mathbb E \po E_r f_{t-r}(Z_r) \pf &= & \partial_r \mathbb E \po h(r,Z_r,E_r)\pf \\
&= & 
\mathbb E \po  \po \partial_r + \mathcal L_r\pf  h(r,Z_r,E_r)   \pf \\
&= & 
\mathbb E \po E_s \po - \partial_r + A_{t-r}  + L \pf  f_{t-r}(Z_r)   \pf \\
&= & 
-\mathbb E \po E_r c_{t-r}(Z_r)   \pf .
\end{eqnarray*}
Using that $E_0=I$,  conclusion follows from
\[ f_t(z) = \mathbb E \po E_0 f_t(Z_0)\pf = \mathbb E \po E_t f_0(Z_t)\pf - \int_0^t \partial_r \mathbb E \po E_r f_{t-r}(Z_r) \pf \dd r\,.  \]
\end{proof}

\subsection{Proof of Lemma~\ref{lem:uniqueIPM}}
We can write $\sigma^\varepsilon(t) = \sigma(t/\varepsilon)$, and then drop the superscript $\varepsilon$ to alleviate notations; in other words, without loss of generality it is sufficient to treat the case $\varepsilon=1$ (up to replacing $\gamma$ by $\gamma/\varepsilon$ in Assumption~\ref{assu:settings}).

For $s \in \mathcal{S}$, let $F_s(x) = x(a_{10}^s - a_{11}^s x)$. By definition of $p_0$ and $p_1$, we have that for all $s\in \mathcal{S}$, $F_s(x) > 0$ for $x \in (0,p_0)$ and $F_s(x) < 0$ for $x > p_1$, which implies that for all $X_0 \neq 0$, there exists a time $T_0$ such that, for all $t \geq T_0$, $X_t \in [p_0, p_1]$. Now, on the state space $E= [p_0, p_1] \times \mathcal{S}$, we introduce the metric
\[
d((x,s),(y,s')) = 1_{ s \neq s'} + 1_{s = s'} \frac{| x - y|}{p_1 -p_0},
\]
and we consider on the space of probability measures on $E$ the associated Wasserstein distance defined for all $\mu, \nu$ by 
\[
\mathcal{W}_d ( \mu, \nu) = \inf_{\Gamma} \int d((x,s), (y,s')) d\,\Gamma( (x,s), (x,s')),
\]
where the infimum runs overs all the coupling measures of $\mu$ and $\nu$. We will show that there exists $C, \delta > 0$ such that, for all $(x,s), (y,s') \in E$ and all $t \geq 0$,
\begin{equation}
\label{eq:cvWasserstein}
\mathcal{W}_d ( \delta_{(x,s)} P_t, \delta_{(y,s')} P_t ) \leq C e^{ - \delta t} d( (x,s), (y,s')),
\end{equation}
where $P_t$ stands for $P_t^{1}$, the semigroup of $(X, \sigma)$ at  time $t$. The above inequality implies Lemma~\ref{lem:uniqueIPM} by classical arguments, since the space of probability measures on $E$ endowed with the Wasserstein metric $\mathcal{W}_d$ is complete.

We first prove~\eqref{eq:cvWasserstein} for point $(x,s)$ and $(y,s')$ such that $s'=s$. In that case, a coupling of $\delta_{(x,s)} P_t$ and $\delta_{(y,s)} P_t$ is given by the random variables $((X_t^{x,s}, \sigma^s_t),(X_t^{y,s}, \sigma^s_t))$, where for all $T \geq 0$,
\[
X_T^{x,s} = x + \int_0^T F_{\sigma^s(u)}(X_u^{x,s}) du, \quad  X_T^{x,s} = y + \int_0^T F_{\sigma^s(u)}(X_u^{y,s}) du.
\]
In other words,  $X_t^{x,s}$ and $X_t^{y,s}$ denote the solution to~\eqref{eq:Lotka1} with respective initial conditions $x$ and $y$ and constructed with the same realisation of the processus $(\sigma(u)_{u \geq 0}$ starting from $s$.

\begin{lem}
Almost surely,  for all $x,y \in [p_0, p_1]$, $s \in \mathcal{S}$ and $t\geqslant0$,
\begin{equation}
\label{eq:couplagePS}
    | X_t^{x,s} - X_t^{y,s} | \leq \frac{p_1}{p_0} e^{ - \eta t} | x - y|,
\end{equation}
where $\eta = p_0 \inf_{s \in \mathcal{S}}a_{11}^s > 0$.  This immediately yields
\[
\mathcal{W}_d ( \delta_{(x,s)} P_t^{\varepsilon}, \delta_{(y,s)} P_t^{\varepsilon} ) \leq \frac{p_1}{p_0} e^{ - \eta t}  d( (x,s), (y,s)).
\]
\end{lem}

\begin{proof}
Let $s \in \mathcal{S}$ and write $\alpha_t = a^{\sigma^s(t)}_{10}$ and $\beta_t = a^{\sigma^s(t)}_{11}$. Let $\varphi_t(z)$ the solution of the differential equation
\begin{equation}
\label{eq:varlog}    
\dot z_t = \alpha_t - \beta_t e^{z_t}, \quad z_0=z
\end{equation}
Then, for all $z,z' \geq \ln(p_0)$, and all $t \geq 0$, we have
\begin{equation}
\label{inq:varlog}    
|\varphi_t(z) - \varphi_t(z')|\leq e^{ - p_0 \int_0^t \beta_u du} |z - z' |.
\end{equation}
Indeed, the associated variational equation is 
\[
\frac{d}{dt}( \partial_z \varphi_t(z) ) = - \beta _u( \partial_z \varphi_t(z)) e^{ \varphi_t(z)},
\]
thus 
\begin{align*}
    \partial_z \varphi_t(z) & = \exp \left( - \int_0^t \beta_u e^{ \varphi_u(z)} du \right)\\
    & \leq  \exp \left( - p_0 \int_0^t \beta_u  du \right),
\end{align*}
where the inequality comes from the fact that  $z \geq \ln(p_0)$ implies $\varphi_t(z) \geq \ln(p_0)$ for all $t \geq 0$.

Now, consider the differential equation 
\begin{equation}
\label{eq:LValpha}
    \dot x_t = x_t ( \alpha_t - \beta_t x_t), \quad x_0 = x
\end{equation}
and set $z_t = \ln(x_t)$. Then, $z_t$ solves~\eqref{eq:varlog} with initial condition $\ln(x)$. Denoting by $\psi_t(x)$ the solution to~\eqref{eq:LValpha}, it holds true, using~\eqref{inq:varlog}, that for all $x,y \in [p_0, p_1]$ and all $t \geq 0$,
\begin{align*}
    | \psi_t(x) - \psi_t(y) | & = | e^{\varphi_t(\ln(x))} - e^{\varphi_t(\ln(y))}|\\
    & \leq p_1 |\varphi_t(x) - \varphi_t(y)|\\
    & \leq p_1 e^{ - p_0 \int_0^t \beta_u du}| \ln(x) - \ln(y) |\\
    & \leq \frac{p_1}{p_0} e^{ - p_0 \int_0^t \beta_u du}|x - y|.
\end{align*}
This entails the result since $X_t^{x,s}$ and $X_t^{y,s}$  solve~\eqref{eq:LValpha} with respective initial conditions $x$ and $y$.
\end{proof} 
We now tackle the case where the initial conditions of  $\sigma$ in the two chains are distinct $s\neq s' \in\mathcal S$. Let $t>0$. Thanks to Assumption~\ref{assu:settings}, there exists a coupling $(\sigma_{t/2}^s,\sigma_{t/2}^{s'})$ of $\delta_{s} \mathcal Q_{t/2}$ and $\delta_{s'} \mathcal Q_{t/2}$ with
\[\mathbb P \po \sigma_{t/2}^{s} \neq \sigma_{t/2}^{s'}\pf \leqslant C e^{-\gamma t}\,.\]
Fix $x,y\in [p_0,p_1]$, and let $X_{t/2}^{x,s}$ be such that $(\sigma_{t/2}^s,X_{t/2}^{x,s})$ is distributed according to $\delta_{x,s}P_{t/2}^{\varepsilon}$ (which is possible since regular version of the conditional laws exists in Polish spaces, see e.g. \cite[Theorem 5.3]{Kallenberg}), and similarly for $X_{t/2}^{y,s'}$. Then, we define $(\sigma_r^s,X_r^{x,s},\sigma_r^{s'},X_r^{y,s'})_{r\geqslant t/2}$ as follows: under the event $\{\sigma_{t/2}^s = \sigma_{t/2}^{s'}\}$, we let $(\sigma_{r}^s)_{r\geqslant t/2}$ be a Markov process associated to $(\mathcal Q_r)_{r\geqslant 0}$ with initial condition $\sigma_{t/2}^{s}$, and we set $\sigma_{r}^{s'}=\sigma_r^{s}$ for all $r\geqslant t/2$; otherwise, we define these processes as two independent processes associated to $(\mathcal Q_r)_{r\geqslant 0}$ with respective initial conditions $\sigma_{t/2}^s$ and $\sigma_{t/2}^{s'}$. In both cases, we define $(X_r^{x,s},X_r^{y,s'})_{r\geqslant t/2}$ as solutions to
\[
X_r^{x,s} = X_{t/2}^{x,s} + \int_{t/2}^r F_{\sigma^s(u)}( X_u^{x,s}) du,  \quad X_r^{y,s'} = X_{t/2}^{y,s'} + \int_{t/2}^r F_{\sigma^{s'}(u)}( X_u^{y,s'}) du.
\]
Then $(\sigma_t^s,X_t^{x,s},\sigma_t^{s'},X_t^{y,s'})$ is a coupling of $\delta_{(x,s)} P_t^{\varepsilon}$ and $ \delta_{(y,s)} P_t^{\varepsilon}$ with, using \eqref{eq:couplagePS} and that $|X_{t/2}^{x,s} - X_{t/2}^{x,s} | \leqslant p_1-p_0 $ almost surely,  
\begin{eqnarray*}
    \lefteqn{\mathbb{E}\left(d ( (X_t^{x,s}, \sigma_t^s),  (X_t^{y,s'}, \sigma_t^{s'}))\right)}\\
    & = &  \mathbb P \po \sigma_{t/2}^{s} \neq \sigma_{t/2}^{s'}\pf   + \mathbb{E}\left( \frac{| X_t^{x,s} - X_t^{y,s'}|}{p_1-p_0}  1_{\sigma_{t/2}^{s} = \sigma_{t/2}^{s'}}\right)\\
    & \leq &  C e^{ - \gamma t/2} + \frac{p_1}{p_0} e^{- \eta t/2}\\
    & \leq & \po C+ \frac{p_1}{p_0} \pf e^{-(\gamma \wedge \eta) t/2} d\po (x,s),(y,s')\pf\,.
\end{eqnarray*}
As a conclusion, we have obtained that \eqref{eq:cvWasserstein} holds for all $(x,s),(y,s') \in [p_0,p_1]\times \mathcal S$ for some constants $C,\delta>0$, which concludes the proof. 

\subsection*{Acknowledgements}

The research of P. Monmarché is supported by the French ANR grant SWIDIMS (ANR-20-CE40-0022).

\bibliographystyle{plain}
\bibliography{biblio}
\end{document}